\documentclass[12pt]{article}

\usepackage{amsmath,amscd,amsfonts,amsthm, xypic, amssymb, graphicx, baskervald, geometry, todonotes, cleveref}
\input{xy}
\usepackage{color}

\newcommand{\shrinkmargins}[1]{
  \addtolength{\textheight}{#1\topmargin}
  \addtolength{\textheight}{#1\topmargin}
  \addtolength{\textwidth}{#1\oddsidemargin}
  \addtolength{\textwidth}{#1\evensidemargin}
  \addtolength{\topmargin}{-#1\topmargin}
  \addtolength{\oddsidemargin}{-#1\oddsidemargin}
  \addtolength{\evensidemargin}{-#1\evensidemargin}
  }

\shrinkmargins{.7}

\DeclareMathOperator{\Diff}{Diff}

\DeclareMathOperator{\Homeo}{Homeo}

\DeclareMathOperator{\Hom}{Hom}
\DeclareMathOperator{\Map}{Map}

\DeclareMathOperator{\SL}{SL}

\DeclareMathOperator{\GL}{GL}
\DeclareMathOperator{\gl}{gl}
\DeclareMathOperator{\ch}{ch}
\DeclareMathOperator{\join}{join}
\DeclareMathOperator{\Sp}{Sp}

\DeclareMathOperator{\Mod}{Mod}

\DeclareMathOperator{\Th}{Th}

\DeclareMathOperator{\U}{U}
\DeclareMathOperator{\SU}{SU}

\DeclareMathOperator{\spinc}{spin_{\C}}

\newcommand{\field}[1]{\mathbb{#1}}

\newcommand{\Q}{\field{Q}}

\newcommand{\Z}{\field{Z}}
\newcommand{\bZ}{\field{Z}}
\newcommand{\A}{\field{A}}
\newcommand{\ka}{\mathfrak{a}}
\newcommand{\KA}{\mathfrak{A}}
\newcommand{\KS}{\mathfrak{S}}

\newcommand{\C}{\field{C}}
\newcommand{\Ehat}{\hat{E}}
\newcommand{\ehat}{\hat{e}}
\newcommand{\deltahat}{\hat{\delta}}

\newcommand{\Fhat}{\hat{F}}

\newcommand{\xhat}{\hat{x}}
\newcommand{\ihat}{\hat{i}}
\newcommand{\Hhat}{\hat{H}}

\newcommand{\pihat}{\hat{\pi}}
\newcommand{\phat}{\hat{p}}
\newcommand{\te}{\widetilde{e}}

\newcommand{\beq}{\begin{displaymath}}
\newcommand{\eeq}{\end{displaymath}}
\newcommand{\beqn}{\begin{equation}}
\newcommand{\eeqn}{\end{equation}}

\newcommand{\st}{\textsuperscript{st}}

\providecommand{\arr}{\to}

\providecommand{\Perf}{\mathrm{Perf}}
\providecommand{\Line}{\mathrm{Line}}
\providecommand{\sma}{\wedge}

\providecommand{\Cat}{\cC\mathrm{at}}
\providecommand{\iCat}{\Cat_{\infty}}
\providecommand{\Sp}{\mathcal{S}\mathrm{p}}
\providecommand{\St}{\mathcal{S}\mathrm{t}_{\infty}}

\providecommand{\Vect}{\mathcal{V}\mathrm{ect}}

\providecommand{\Fun}{\mathrm{Fun}}
\providecommand{\Mot}{\mathrm{Mot}}
\providecommand{\ex}{\mathrm{ex}}

\newcommand{\cA}{\mathcal{A}}\newcommand{\cB}{\mathcal{B}}\newcommand{\cC}{\mathcal{C}}\newcommand{\cU}{\mathcal{U}}\newcommand{\cV}{\mathcal{V}}

\theoremstyle{plain}
\newtheorem{introthm}{Theorem}

\theoremstyle{plain}
\newtheorem{thm}[equation]{Theorem}
\newtheorem{prop}[equation]{Proposition}
\newtheorem{cor}[equation]{Corollary}
\newtheorem{lem}[equation]{Lemma}

\theoremstyle{definition}
\newtheorem{defn}[equation]{Definition}

\newtheorem{exmp}[equation]{Example}

\newtheorem{question}[equation]{Question}
\newtheorem{rem}[equation]{Remark}

\theoremstyle{remark}

\title{Twisted iterated algebraic $K$-theory and topological T-duality for sphere bundles}
\author{John A. Lind, Hisham Sati, and Craig Westerland}

\begin{document}
\bibliographystyle{amsalpha}

\maketitle
\begin{abstract} We introduce a periodic form of the iterated algebraic K-theory of $ku$, the (connective) complex K-theory spectrum, as well as a natural twisting of this cohomology theory by higher gerbes.  Furthermore, we prove a form of topological T-duality for sphere bundles oriented with respect to this theory.\end{abstract}

\tableofcontents

\vspace{1cm}
Let $ku$ be the connective complex K-theory spectrum.  The underlying infinite loop space 
$\Omega^\infty ku = \Z \times BU$ of $ku$ classifies virtual complex vector bundles.  The cohomology theory associated to the algebraic K-theory spectrum $K(ku)$, the subject of much recent research in homotopy theory \cite{AR1}, has a geometric interpretation as a Grothendieck group of 2-vector bundles \cite{baas-dundas-rognes, baas-dundas-richter-rognes, Lind}.
A 2-vector bundle is a bundle whose fiber is a 2-vector space, which is a categorified form of a vector space introduced by Kapranov-Voevodsky \cite{KV}. 
Forming equivalence classes of 2-vector bundles over $X$ leads to a bimonoidal category
$2{\rm Vect}(X)$.  By \cite{baas-dundas-richter-rognes}, the Grothendieck group completion of $2{\rm Vect}(X)$ is represented by the infinite loop space $\Omega^{\infty}K(ku)$ underlying the algebraic $K$-theory of $ku$.

\medskip
Applying the functor $K(-)$ again, 
one is naturally led to imagine that the iterated algebraic $K$-theory spectrum 
$$
\ka_n := K^{(n-1)}(ku)=\underbrace{K(K( \cdots K}_{n-1}(ku)\cdots ))
$$ 
has an interpretation in terms of categorified bundles native to $n$-category theory.  It is expected that algebraic K-theory in many cases increases chromatic complexity by one,
i.e., that it produces a constant ``red-shift'' by one chromatic layer in stable homotopy theory \cite{AR1}.  In this paper, we study a Bott-periodic form $\KA_n := K^{(n-1)}(ku)[\beta_{n}^{-1}]$ of iterated algebraic $K$-theory.  While our results do not provide direct evidence either for or against the Ausoni-Rognes red-shift conjectures \cite{AR2}, our interest in the relationship between the geometric content of iterated algebraic $K$-theory and chromatic homotopy theory is a primary motivation for the study of T-duality in $K^{(n-1)}(ku)[\beta_{n}^{-1}]$-theory.


\medskip

Much as line bundles are the fundamental building blocks of vector bundles, and hence play an essential role in the K-theory of vector bundles, $(n-1)$-gerbes are the simplest forms of $n$-vector bundles.  
For a general definition of $n$-gerbes as $n$-truncated and $n$-connected objects, see  \cite[Sec. 7.2.2]{LurieHTT}.
For example, when $n = 2$, a 1-gerbe (also known as a gerbe with band $U(1)$) gives rise to a rank one 2-vector bundle.  This fact is witnessed at the level of classifying spaces by a map
\[
K(\Z, 3) \to B \GL_1 (ku) \to \Omega^{\infty}K(ku).
\]
This is the 2-categorical analog of the map $\C P^\infty \to \Z \times BU$ representing the inclusion of line bundles into the Grothendieck group of vector bundles.  The adjoint map $\Sigma^\infty \C P^\infty_+ \to ku$ is a map of $E_\infty$ ring spectra.  We study here a family of analogous $E_\infty$ ring maps
\beqn \Sigma^\infty K(\Z, n+1)_+ \to \ka_n:= K^{(n-1)}(ku) \label{equation_1} \eeqn
which we think of as representing the inclusion of $(n-1)$-gerbes into the Grothendieck group of $n$-vector bundles.  

\medskip
These maps are adjoint (under the adjunction described in \cite{ABGHR, May}) to maps of spectra $\Sigma^{n+1} H \Z \to \gl_1(\ka_n)$ or, equivalently, maps of $E_{\infty}$ spaces $K(\Z, n+1) \to \GL_1(\ka_n)$.  Delooping once, we obtain $E_\infty$-twistings of the cohomology theory $\ka_n$ by $(n+2)$-dimensional cohomology\footnote{A caveat is in order: the twisted cohomology group $\ka_n^*(X; H)$ depends upon the representative map $H: X \to K(\Z, n+2)$.  Just as is the case for twisted K-theory, homotopic maps (i.e., cohomologous classes) yield isomorphic twisted cohomology groups; however, the isomorphism is not canonical.} classes: each class $H \in H^{n+2}(X; \Z)$ in the cohomology of a topological space $X$ gives rise to twisted cohomology groups $\ka_n^*(X; H)$.  When $n=1$, this returns the usual notion of connective complex K-theory twisted by a gerbe (or rather, a representative of its Dixmier-Douady class in $H^3$). The new twists that we study have the same degree as the twists of Morava K-theory and E-theory studied previously \cite{SW}; we hope that further understanding of the redshift conjecture will relate the results of that paper and this one.

\medskip
Notice that the element $\beta \in \pi_2 ku$, given as the composite $S^2 \to \C P^\infty \to \Z \times BU$ of the fundamental class of $\C P^\infty = K(\Z, 2)$ and its inclusion as $BU(1) \times \{1\}$, is the Bott class.  Localizing $ku$ at this class yields periodic complex $K$-theory: $KU := ku[\beta^{-1}]$.  Our first goal is an analogue of this construction for $\ka_n$:

\begin{introthm} \label{thm:non-nilpotent}

The composite of the fundamental class of $K(\Z, n + 1)$ with the map in Equation \eqref{equation_1} yields an essential class $\beta_n \in \pi_{n+1} \ka_n$.  Further, when $n$ is odd, $\beta_n$ is non-nilpotent, and none of its powers are torsion.

\end{introthm}

This is proven as Corollary \ref{non_tor_cor} and Theorem \ref{non_nilp_thm}, below.  For odd $n$, we will notate the localized spectrum $\ka_n[\beta_n^{-1}]$, i.e. the Bott-inverted iterated algebraic K-theory, by $\KA_n$. Theorem \ref{thm:non-nilpotent} suggests that this is a non-trivial object of study.  However, it is a consequence of a theorem of Arthan \cite{arthan} that $\KA_n$ is a \emph{rational} spectrum whenever $n>1$.  We may regard it as a higher categorical analogue of periodic topological K-theory, or perhaps an \'{e}tale form of iterated K-theory. 

\medskip
The twisting of $\ka_n$ by cohomology classes in degree $n+2$ extends in a natural way to a twisting of $\KA_n$ via the composite
$$\Sigma^{n+1} H \Z \to \gl_1 \ka_n \to \gl_1 \KA_n,$$
where the second map is induced by the $E_\infty$ ring map given by localization $\ka_n \to \KA_n$.  Our second main purpose in this paper is to prove an analogue of the topological T-duality isomorphism of \cite{BEM, BS} in this context.

\medskip
Given a base space $X$, Bouwknegt, Evslin and Mathai give a criterion for two sets of data $(E, H)$ and $(\Ehat, \Hhat)$ to 
be \emph{T-dual} \cite{BEM}.  Here $E$ and $\Ehat$ are principal $S^1$-bundles over $X$, and $H$ and $\Hhat$ are cocycles representing cohomology classes 
in $H^3(E)$ and $H^3(\Ehat)$, respectively.  The criterion in \cite{BEM} is given in terms of relations among various characteristic classes.  In \cite{BS}, Bunke-Schick reinterpret this criterion as amounting to the existence of a Thom class on a certain $S^3$-bundle over $X$ into which both $E$ and $\Ehat$ embed.  When $(E, H)$ and $(\Ehat, \Hhat)$ are a T-dual pair, there is an isomorphism of twisted $K$-theory groups $K^*(E; H) \cong K^{* - 1}(\Ehat, \Hhat)$ called the T-duality isomorphism.

\medskip
Let $E$ and $\Ehat$ be fiber bundles over $X$ with fiber the $q$-sphere $S^q$, and let $H$ and $\Hhat$ be classes in $H^{2q+1}(E)$ and $H^{2q+1}(\Ehat)$, respectively.  In Definition \ref{t-dual_defn}, we will give a Thom class criterion for $(E, H)$ and $(\Ehat, \Hhat)$ to be T-dual in the higher (dimensional and categorical) context, similar to Bunke-Schick's.  We then prove the following extension of the T-duality isomorphism in this setting:

\begin{introthm} \label{thm_b}

Let $n=2q-1$, and assume that $(E, H)$ and $(\Ehat, \Hhat)$ are a T-dual pair.  Let $\tau_{E} \colon E \to B \GL_1 \KA_n$ denote the orientation twisting determined by the vertical tangent bundle of $E \to X$.  Then there is an isomorphism of twisted cohomology groups
$$T:= \phat_! \circ \Lambda \circ p^*: \KA_n^*(E; \tau_{E} \otimes H) \to \KA_n^{*}(\Ehat; \Hhat)$$
given in terms of a Fourier-Mukai push-pull construction on the correspondence space $E \times_X \Ehat$.  Given an $\KA$-orientation of the fiber bundle $E \to X$, the T-duality isomorphism takes the form
$$T : \KA_n^{* + q}(E; H) \to \KA_n^{*}(\Ehat; \Hhat).$$

\end{introthm}

This is proven as Theorem \ref{t_dual_thm}.  In fact, we prove the result for a larger class of cohomology theories $R$ than just $\KA_n$, namely those which may be twisted by $n$-gerbes in such a fashion that the analogue of the class $\beta_n$ is invertible in $\pi_* R$.  One consequence of this invertibility assumption (again via \cite{arthan}) is that $R$ is a rational spectrum when $n>1$.  We will show in Theorem \ref{rational_thm} that this is no accident: every cohomology theory $R$ for which the T-duality map $T$ of the previous theorem is an isomorphism must be rational.

\medskip Additionally, we study criteria to ensure the orientability requirement of the previous result, and analyze the homotopy type of the classifying spaces for T-dual pairs, much as in \cite{BS} in the case $q=n=1$. 

\medskip
For $q=3$, T-duality for $S^3$-bundles in 
rational cohomology, and in twisted K-theory under
some conditions on the cohomology and dimensions of the underlying manifolds,
was considered recently in \cite{BEM2}.  
The authors discuss twistings of topological K-theory by a 7-dimensional class.  It is known that $B{\rm GL}_1(ku)$ splits as
$B\Z/2 \times K(\Z, 3) \times BB{\rm SU}_{\otimes}$.  Furthermore, $BB{\rm SU}_{\otimes}$ has a 7-dimensional homotopy group isomorphic to $\Z$; this is visible via the
$k$-invariant $BB{\rm SU}_{\otimes}\langle 6 \rangle \to K(\Z, 7)$ in a Postnikov tower for $BB\SU_\otimes$.
An essential map in the other direction -- which is necessary for there to be a twist of K-theory by $H^7$ -- does not exist (see \cite{AGG}); hence that degree cannot be `isolated' in general, 
as highlighted in \cite{Higher} and clarified further in \cite{SW}. 
The conditions imposed by \cite{BEM2} evade the obstructions  
for dimension reasons and
 allow one to isolate such a twist in special cases. 
If one rationalizes, as done for the most part in \cite{BEM2}, one gets periodic cohomology, which does have a degree seven
 twist.  We recover this result and analogues for all sphere bundles as a consequence of the variant of Theorem \ref{thm_b} for periodic rational cohomology.  Indeed, it is a consequence of Theorem \ref{rational_thm} that such a T-duality result \emph{must} be rational outside of the case $q=1$.

\medskip
Our general result also provides a recipe for determining whether or not there exists a T-dual for $(E, H)$, where $E$
is an $S^q$-bundle with structure group $G \to {\rm Homeo}(S^q)$ equipped with an $n$-gerbe $H$, and whether or not the T-dual is unique. 
This is determined by analyzing the connectivity of the Euler class map $BG \to K(\Z, q+1)$. 
For the case $q=1$, this recovers some of the results of Baraglia \cite{Ba14} and 
 Mathai-Rosenberg \cite{MR14} for existence and uniqueness of 
T-dual bundles for non-principal circle bundles.

\medskip
The paper is organized as follows. 
In Sec. \ref{Sec Gen},
we describe the multiplicative behavior of the algebraic $K$-theory functor and consider generalities about twisted cohomology theories. 
In Sec. \ref{sec:iterated_twistings_ku}, we construct and analyze the twisting of iterated algebraic $K$-theory $K^{(n - 1)}(ku)$ by higher gerbes.  We prolong some computations of the homotopy of $K(ku)$ due to \cite{adr} to the higher setting in section \ref{rel_sec}.
Then in Sec. \ref{Sec Snaith} we consider a higher analog of Snaith's construction of $KU$,
which in turn admits the universal periodic twist.

In order to obtain concrete expressions and with an eye for applications to 
T-duality, we study Chern characters and orientations in Sec. 
\ref{Sec Chern}. 
We introduce two Chern characters  associated to the new twisted theories,
one in Sec. \ref{Sec Dold}
and Sec. \ref{Sec higherChern}, and then we describe orientations with respect
the new theories in Sec. \ref{Sec Orient}.  

The general set-up and the proof of the main T-duality theorem and its converse are given in 
Sec. \ref{Sec T}.  Sec.  
\ref{Sec pairs} is concerned with the question of the existence and the uniqueness of T-dual pairs, which we analyze homotopy theoretically in terms of classifying spaces. In particular, in Sec.  \ref{Sec P}
we introduce a space $P_n(G)$ which classifies the possible T-dual pairs. Then in Sec. \ref{Sec R}
we classify, via a space $R_n(G)$, bundles and Thom classes that can arise in our context of 
T-duality. The two constructions are related in Sec. \ref{Sec_Compare}
by constructing a forgetful map from $P_n(G)$ to $R_n(G)$, where we also
show how our results reproduce earlier results on T-duality. 

The final section
Sec. \ref{Sec ncat} is a speculative account of the relationship between the iterated algebraic 
K-theory considered earlier in the paper and a proposed model for the algebraic $K$-theory of $n$-vector spaces defined in terms of enriched higher category theory.

\medskip
{\it Acknowledgements.}  We thank Ben Antieau, Justin Noel, and Thomas Schick for helpful comments on a previous draft and Drew Heard for pointing us to the paper of Arthan \cite{arthan}.  We also thank Rune Haugseng for sharing his perspectives on higher category theory and iterated $K$-theory with us.  JL was supported in part by the DFG through SFB1085. 
 HS thanks the Erwin Schr\"odinger Institute for Mathematical Physics, Vienna, for hospitality and the organizers of the  
 program ``Higher Structures in String Theory" for the opportunity to present the results of this project.  CW was supported in part by the NSF through DMS-1406162.

\section{Iterated algebraic K-theory and its twistings by $n$-gerbes}
\label{Sec Alg}

We recall here the multiplicative behavior of the algebraic $K$-theory functor and describe the twisting of iterated algebraic $K$-theory by higher gerbes.  

\subsection{Generalities on twistings of algebraic K-theory}
\label{Sec Gen}

Given an $E_\infty$ ring spectrum $A$, the algebraic K-theory spectrum $K(A)$ is once again an $E_\infty$ ring spectrum.  There is a natural map
$B\GL_1 A \to \Omega^\infty K(A)$
coming from the inclusion of $A$-lines into all cell $A$-modules.  It is not the case that this is an infinite loop map: the multiplication on $B \GL_1 A$ is by tensor products of $A$-lines, whereas that in $\Omega^\infty K(A)$ is from the sum of modules.  
However, this map has image in $\GL_1 K(A)$ (since $A$-lines are invertible $A$-modules), and in fact the induced map $B\GL_1 A \to \GL_1 K(A)$ is an infinite loop map.

\medskip
We explain the above in detail. Let $A$ be a connective commutative ring spectrum and let $K(A)$ denote the connective algebraic $K$-theory spectrum of $A$.  The underlying infinite loop space of the spectrum $K(A)$ receives a map
\begin{equation}\label{eq:pre_w}
B \GL_1 A \to \coprod_{n \geq 0} B \GL_n A \to \Omega^{\infty}K(A)
\end{equation}
from the classifying space of the space of units $\GL_1 A$ via the classifying space of finite rank free $A$-module spectra.
This is the analog in algebraic $K$-theory of the map $\C P^\infty \to \Z \times B \U$ into topological $K$-theory classifying the homomorphism from the Picard group into the Grothendieck group of vector bundles.  As in the classical case, the source and target both inherit $E_{\infty}$ space structures from the multiplication on $A$.

\begin{prop} \label{prop_3}
The map \eqref{eq:pre_w} has image in $\SL_1 K(A)$; the result is a map of $E_{\infty}$ spaces and so lifts to a map of spectra
\[
\mu \colon \Sigma \gl_1 A \to \gl_1 K(A).
\]
\end{prop}

\begin{proof}
We employ the $\infty$-categorical model for algebraic $K$-theory developed by Gepner-Groth-Nikolaus \cite{GGN} using the language of quasicategories.  Let $\Perf_{A}$ denote the stable $\infty$-category of compact $A$-module spectra.  Let $\Line_{A}$ denote the $\infty$-category of rank one $A$-module spectra.  Both of these $\infty$-categories admit symmetric monoidal structures under the smash product $\sma_{A}$ of $A$-modules.  Writing $\iota_0$ for the groupoid core functor which takes a quasicategory to its maximal sub Kan complex, the inclusion of $\Line_{A}$ into $\Perf_{A}$ induces a symmetric monoidal functor $i \colon  \iota_0 \Line_{A} \to \iota_0 \Perf_{A}$ of symmetric monoidal $\infty$-groupoids.  The $\infty$-category $\Perf_{A}$ admits an additional symmetric monoidal structure under the coproduct of $A$-modules and the two monoidal structures combine to give the Kan complex $\iota_0 \Perf_{A}$ the structure of an $E_{\infty}$ ring space \cite[Corollary {8.11}]{GGN}.  The algebraic $K$-theory $K(A)$ is the connective $E_{\infty}$ ring spectrum for which $\Omega^{\infty}K(A)$ is the group completion of $\iota_0 \Perf_{A}$ as an $E_{\infty}$ ring space 
\cite[Definition {8.3} and Theorem 8.6]{GGN}.  The resulting composite
\[
\iota_0 \Line_{A} \to \iota_0 \Perf_{A} \to \Omega^{\infty} K(A)
\]
is a map of $E_{\infty}$ spaces, where we use the multiplicative structure on $\iota_0 \Perf_{A}$ and $\Omega^{\infty}K(A)$.  This map is a model for the map \eqref{eq:pre_w}, as can be verified along the lines of  \cite[Proposition 2.9]{ABGHR2}, so the map in question is $E_{\infty}$.

The source of the map is a connected Kan complex and its image lies in the component of the unit object 
$A \in \iota_0 \Perf_{A}$, hence in the component $\SL_1 K(A)$ of the unit in the $E_{\infty}$ ring space $\Omega^{\infty} K(A)$.  In particular, the map factors as an $E_{\infty}$ map through the space $\GL_1 K(A)$ of units, which deloops to the map of spectra $\mu$.
\end{proof}

\medskip

We now recall how to construct twisted forms of the cohomology theory represented by the ring spectrum $A$, following \cite{ABG, ABG2, ABGHR}.  Given a space $X$ and a map $\omega \colon X \to B \GL_1 A$, let $E \to X$ be the $\GL_1(A)$-fibration pulled back from $B\GL_1 A$ via $\omega$.  Define an $A$-module spectrum
$$\A(X; \omega) := A \wedge^{L}_{\Sigma^\infty \GL_1 A_+} \Sigma^\infty E_+.$$
This is the Thom spectrum associated to the parametrized spectrum of $A$-lines over $X$ classified by $\omega$.  

\begin{defn} 

The {\it $\omega$-twisted $A$-homology and $A$-cohomology groups of} $X$ are defined by
$$A_q(X; \omega) := \pi_q \A(X; \omega) \quad \mbox{ and } \quad A^q(X; \omega) := \pi_{-q} F_A(\A(X; -\omega), A).$$
\noindent Notice that in the definition of twisted cohomology we use the inverse twist $-\omega$ under the monoidal structure induced by the smash product of invertible $A$-modules.  Sometimes the opposite convention is used, but we prefer this choice because it ensures that twisted cohomology agrees with homotopy classes of sections of the parametrized spectrum associated to $\omega$ \cite{Lind}, and conforms with grading conventions when $\omega$ carries a topological dimension, for example when it arises from a classifying map for a virtual vector bundle via the J-homomorphism.  

\end{defn}

Suppose that $h$ is a spectrum, and that $\tau: h \to \Sigma \gl_1 A$ is a map of spectra.  We regard the map $\tau$ as an \emph{$E_\infty$ twisting of $A$}, because the infinite loop map $\Omega^{\infty} \tau: \Omega^\infty h \to B \GL_1 A$ allows us to twist the $A$-cohomology of a space $X$ by elements $[H] \in h^0(X)$.  To do so, we represent the class $[H]$ as a map $H: X \to \Omega^\infty h$ and define the $H$-twisted cohomology of $X$ to be the twisted $A$-homology and $A$-cohomology groups associated to the twist $\Omega^{\infty} \tau \circ H \colon X \to B\GL_1 A$, i.e. we make the abbreviation $A^q(X; H) = A^q(X; \Omega^{\infty} \tau \circ H)$.  Different representatives for the same cohomology class give isomorphic twisted cohomology groups, so the twisted cohomology theory associated to $H$ only depends on the underlying cohomology class $[H]$, but only up to non-canonical isomorphism.

\medskip
We can use the map $\mu \colon \Sigma \gl_1 A \to \gl_1K(A)$ considered in Proposition \ref{prop_3} to construct $E_\infty$ twists of the algebraic $K$-theory spectrum $K(A)$ from (shifts of) $E_\infty$ twists of $A$, and we would like to know whether or not the twists constructed in this way
are essential, i.e. homotopically nontrivial. 

\begin{prop} \label{iter_prop}

For every $E_\infty$ twisting $\tau: h \to \Sigma \gl_1 A$, the composite
$$\Sigma(\mu \circ \tau): \Sigma h \to \Sigma \gl_1 K(A)$$
is an $E_\infty$ twisting of $K(A)$.  Further, if the rationalization of $\tau$ is essential, so too is the rationalization of $\Sigma(\mu \circ \tau)$.

\end{prop}

\begin{proof}

Only the latter statement needs to be proved.  One can see this via the rational determinant $\det_\Q$, defined by Ausoni-Rognes in \cite[Proposition 5.4]{AR}.  This is a map of spaces
$${\det}_\Q: \Omega^\infty K(A) \to (B\GL_1 A)_\Q$$
whose composite with $\Omega^\infty \mu$ is the rationalization map of $B\GL_1 A$.  Thus $\Omega^\infty \Sigma(\mu \circ \tau)$ factorizes the rationalization of $\Omega^\infty \tau$, which yields the result. 

\end{proof}

\begin{rem}  There is not an integral determinant map lifting $\det_{\Q}$.
In fact, the obstruction to the existence of a continuous map $\Omega^\infty K(ku) \to B{\rm GL}_1(ku)$
 with determinant-like properties
 leads to the notion of oriented 2-vector bundles, and gives rise to an oriented version of 
 K-theory of 2-vector bundles with a lift of the natural map from $K(\Z,3)$, called the determinant gerbe map \cite{Kr13}.\end{rem}

\subsection{Twisting the iterated algebraic K-theory of $ku$}
\label{sec:iterated_twistings_ku}

Our main example arises via connective complex topological K-theory, $ku$.  There is a well known $E_{\infty}$ twisting of $ku$ by 3-dimensional cohomology classes; the map 
$$\tau: \Sigma^3 H\Z \to \Sigma \gl_1 ku$$
is the delooping of the map of $E_{\infty}$-spaces $\C P^\infty \to \GL_1 ku$ that regards a complex line as an invertible $\C$-module.

\begin{defn}

Write $\ka_n$ for the {\it iterated algebraic K-theory spectrum} $K^{(n-1)}(ku)$.  Let 
\[
\tau_n: \Sigma^{n+2} H\Z \to \Sigma \gl_1 \ka_n
\]
be the {\it $E_\infty$ twisting of $\ka_n$} obtained by an $(n-1)$-fold iteration of the procedure in Proposition \ref{iter_prop} applied to $\tau = \tau_{1}$.

\end{defn}

Since the original map $\tau$ was rationally essential in homotopy, so too are the maps $\tau_n$:

\begin{cor} \label{non_tor_cor}

The composite of $\tau_n$ and the fundamental class $\iota$ of $\Sigma^{n+2} H\Z$ defines a non-torsion element 
$\tau_n \circ \iota \in \pi_{n+2}\Sigma \gl_1 \ka_n.$

\end{cor}


\begin{defn} We write $\Omega^{\infty + 1}\tau_n \circ \iota \colon S^{n + 1} \to GL_1 \ka_n$ for the map of spaces representing the image of the class $\tau_n \circ \iota$ under the canonical isomorphism  $\pi_{n+2}\Sigma \gl_1 \ka_n \cong \pi_{n + 1} \gl_1 \ka_n$.  The map $\Omega^{\infty + 1}\tau_n$ carries $S^{n + 1}$ into the component of $\GL_1 \ka_n$ corresponding to $1 \in \pi_0 \ka_n$.  Subtracting $1$ gives a based map into $\Omega^\infty \ka_n$, and thus a class $\beta_n \in \pi_{n + 1}\ka_n$.  In other words, we define $\beta_n := [\Omega\tau_n \circ \iota] - 1$, where $[\Omega\tau_n \circ \iota]$ denotes the image of $\Omega\tau_n \circ \iota$ under the isomorphism $\pi_{n + 1} \gl_1 \ka_n \cong \pi_{n + 1} \ka_n$ induced by the inclusion of components $\GL_1 \ka_n \subset \Omega^{\infty} \ka_n$.  Equivalently, $\beta_{n}$ is represented by the composite map of spectra
\begin{equation}\label{eq:beta}
\beta_n \colon \xymatrix@1{S^{n+1} \ar[r]^-{\iota} & \Sigma^\infty K(\Z, n+1)_+ \ar[r]^-{t_n} & \ka_n},
\end{equation}
where $\iota$ is the fundamental class and $t_n$ is adjoint to $\Omega \tau_n$ in the adjunction
$$[\Sigma^\infty K(\Z, n+1)_+, \ka_n]_{E_\infty\mathrm{ring}} \cong [\Sigma^{n+1} H\Z, \gl_1 \ka_n]_{\mathrm{Sp}}.$$
\end{defn}

Note that $\beta_1 \in \pi_2 \ka_1 = \pi_2 ku$ is the usual Bott class which we invert to obtain $KU$.

\begin{thm} \label{non_nilp_thm}

If $n$ is even, $2\beta_n^2=0$.  However, if $n$ is odd then $\beta_n$ is not nilpotent and all powers $\beta_n^m$ are non-torsion, for any positive integer 
$m$.

\end{thm}

\begin{proof}

The first claim is simply that $\beta_n$ is an odd-dimensional element of a graded-commutative ring.  For the second, consider the composite map \eqref{eq:beta} and note that the Hurewicz image of $\iota \colon S^{n+1} \to \Sigma^\infty K(\Z, n+1)_+$ is essential.  In fact, whenever $n$ is odd, the homology ring $H_*(K(\Z, n+1); \Q)$ is the rational  polynomial ring on the Hurewicz image of $\iota$.  In particular, all powers of $\iota$ are rationally essential.

We recall from \cite{Se} that the natural map of ring spectra $\Sigma^\infty \C P^\infty_+ \to ku$ adjoint to $\tau$ is a rational equivalence.  Therefore the iterated Dennis trace map, which is a transformation from algebraic K-theory to topological Hochschild homology $THH$ 
(see \cite{BHM, DGM}), followed by rationalization, may be written as
$$\xymatrix@1{\ka_n = K^{(n-1)}(ku) \ar[r]^-{tr^{n-1}} & THH^{(n-1)}(ku) \ar[r] & THH^{(n-1)}(ku_\Q) & THH^{(n-1)}((\Sigma^\infty \C P^\infty_+)_\Q) \ar[l]_-\simeq}.$$
However, for a loop space $G$, one has $THH(\Sigma^\infty G_+) \simeq \Sigma^\infty LBG_+$.  When $G$ is a $m$-fold loop space, $BG$ (and hence $LBG$) is a $(m-1)$-fold loop space.  We iteratively observe that 
$$THH^{(m)}(\Sigma^\infty G_+) \simeq \Sigma^\infty (LB)^m(G)_{+} \simeq \Sigma^\infty \Map(T^m, B^m G)_{+}$$
if $G$ is connected.  In the second equivalence, we employ the fact that for connected $G$, $LBG \simeq BLG$.  The space of constant functions $T^m \to B^m G$ is homeomorphic to $B^m G$, and so $THH^{(m)}(\Sigma^\infty G_+)$ contains a copy of $\Sigma^\infty B^m G_+$ as a wedge summand.

In the case $G=\C P^\infty = K(\Z, 2)$, it follows that $\Sigma^\infty K(\Z, n+1)_+$ is a summand of $THH^{(n-1)}(\Sigma^\infty \C P^\infty_+)$.  Since rationalization is a smashing localization, 
$$THH(A_{\Q}) \simeq (THH(A))_{\Q}$$
for ring spectra $A$.  Thus the target of the iterated trace map above splits off a wedge factor of $(\Sigma^\infty K(\Z, n+1)_+)_\Q$.  

We claim that the composite of $t_n$ with the iterated trace map and the projection to this wedge summand is the map from $\Sigma^\infty K(\Z, n+1)_+$ to its rationalization; then we may conclude that $t_n$ is an injection on rational homotopy. To see the claim, note that if $G$ is an infinite loop space, there is a homotopy commutative diagram
\[
\xymatrix{
\Sigma^{\infty} BG_{+} \ar[r] \ar[drrr]_-{\mathrm{id}} & \Sigma^{\infty} B \GL_1(\Sigma^{\infty}G_{+})_{+} \ar[r]^-{\eqref{eq:pre_w}} & K(\Sigma^{\infty}G_{+}) \ar[r]^-{tr} &  THH(\Sigma^\infty G_{+}) \simeq \Sigma^\infty LBG_+ \ar[d]^-{ev} \\
& & & \Sigma^{\infty} BG_{+}
}
\]
%
%
which yields Waldhausen's splitting \cite{waldhausen} of $Q(BG_+)$ off of $A(BG) = \Omega^\infty K(\Sigma^\infty G_+)$.   Then the claim follows by iteration (taking $G = K(\Z, m)$ for $m=1, \dots, n$) and rationalization (to accommodate the rational Segal equivalence).

\end{proof}

\begin{defn} 
\label{def U}
 When $n$ is odd, define $\KA_n := \ka_n[\beta_n^{-1}]$. \end{defn}

\noindent One could of course make this construction for $n$ even, but the nilpotence of $\beta_n$ away from the prime 2 in that case will force the resulting spectrum to have 2-torsion homotopy.  In fact, these spectra are contractible when $n$ is even; this follows from Theorem \ref{arthan_thm}, below.

\medskip
The $E_{\infty}$ twisting $\tau_n$ of $\ka_n$ induces an $E_{\infty}$ twisting $K(\Z, n + 2) \to B \GL_1 \KA_n$ by composition with the localization map.  By the phrase ``$n$-gerbe'' we mean a generic term for any geometric structure classified by the Eilenberg-MacLane space $K(\Z, n + 2)$ (for a more general discussion, see \cite[7.2.2.2]{LurieHTT}).  Thus, a complex line bundle is a $0$-gerbe, and a gerbe with band $U(1)$ is a $1$-gerbe.  We say that the $E_{\infty}$ twisting of $\ka_n$ and $\KA_n$ constructed in this section are \emph{twistings by $n$-gerbes}, since an $n$-gerbe over $X$ gives rise to a map $H \colon X \to K(\Z, n + 2)$ by which we can twist these theories.

\subsection{Some relations in the homotopy of $\ka_n$} \label{rel_sec}

Although the following observations will not be used in our work with the periodic spectrum $\KA_n$, they provide the first known results about the homotopy type of iterated algebraic $K$-theory of $ku$. 
In \cite{adr}, it is shown that for $n=2$  there is a class $\zeta \in \pi_3 K(ku)$ with the property that  
$$\beta_2 = 2\zeta-\nu,$$ 
where $\nu$ is the image of the quaternionic Hopf fibration $\nu \in \pi_3(S^0)$ under the unit map to $K(ku) = \ka_2$.  One may iteratively prolong this equation to one in $\pi_{n+1} \ka_n$ in a natural fashion, as we now explain.  Following the recipe given in Proposition \ref{prop_3}, we have  a homomorphism\footnote{Here we identify $\pi_k R = \pi_k \Omega^\infty_0 R$ with $\pi_k \gl_1 R = \pi_k \GL_1 R$ for $k>0$ via the shift in components by 1.}
$$\pi_{n+1} \ka_n \cong \pi_{n+1} \gl_1 \ka_n \to \pi_{n+1} \Sigma^{-1} \gl_1 K(\ka_n)
 \cong \pi_{n+2} \ka_{n+1}\;$$
for each $n > 0$.  By construction, this map carries $\beta_n$ to $\beta_{n+1}$; it also allows us to define classes $Z_n$ and $N_n$ in $\pi_{n+1} \ka_n$ as the iterated images of $\zeta$ and $\nu$.  Since this is a group homomorphism, the relation $\beta_2 = 2\zeta-\nu$ persists:
\begin{prop}
The class $\beta_n$ is given in terms of the iteratively defined classes $Z_n$ and $N_n$ by 
$$\beta_n = 2 Z_n - N_n.$$
\end{prop}
Furthermore, since $24\nu = 0$, we also have $24 N_n = 0$.  Additionally, $\nu$ is nilpotent.  Since the map above is not a ring homomorphism (in fact, it behaves in some sense more like a derivation), we cannot conclude the same for $N_n$.  However, if $N_n$ is actually nilpotent, then a sufficiently large power of $\beta_n$ is 2-divisible (at least after inverting 3).  

\subsection{Higher Snaith spectra}
\label{Sec Snaith}

When $n=1$, Definition \ref{def U} tautologically yields periodic complex K-theory: $\KA_1 = KU$.  In this case, there is an alternative construction of $KU$ due to Snaith \cite{Sn}.  The map $\Sigma^{\infty} K(\Z, 2)_{+}[\iota^{-1}] \to KU$ induced by the localization of the map $t_1 \colon \Sigma^{\infty} K(\Z, 2)_+ \to ku$ studied in the proof of Theorem 
\ref{non_nilp_thm} is an equivalence of spectra.  In this section, we consider generalizations of this construction using the domain of the map $t_n \colon \Sigma^{\infty} K(\Z, n + 1)_+ \to \ka_n$ and discuss their relationship with the twistings from the previous section.  

Guided by Snaith's theorem, we make the following definition:

\begin{defn}
\label{def S}
For $n$ odd, define the \emph{Arthan spectrum} $\KS_n$ to be the localization 
$$\KS_n := \Sigma^\infty K(\Z, n+1)_+[\iota^{-1}]$$
of the suspension spectrum of $K(\Z, n+1)$ at its fundamental class.

\end{defn}

The $E_\infty$ twisting map $t_n: \Sigma^\infty K(\Z, n+1)_+ \to \ka_n$ carries $\iota$ to $\beta_n$, and so descends to a map between the periodic spectra $T_n: \KS_n \to \KA_n$.  Snaith's theorem indicates that this is an equivalence when $n=1$.  When $n>1$, these spectra were studied by Arthan \cite{arthan}, where they were found to display remarkably different behavior:

\begin{thm}[Arthan] \label{arthan_thm} If $n$ is even, $\KS_n$ is contractible.  If $n$ is odd and greater than 1, $\KS_n$ is $(n+1)$-periodic rational cohomology:
$$\KS_n \cong H\Q[\iota^{\pm 1}].$$
\end{thm}

Let $R$ be an $A_\infty$ ring spectrum which receives an $A_\infty$ ring map $\phi: \KS_n \to R$.  The case $R = \KA_n$ is our main example.  From $\phi$, we obtain a natural composite of maps of spaces
$$K(\Z, n+2) \to B\GL_1(\Sigma^\infty K(\Z, n+1)_+) \to B\GL_1(\KS_n) \to B\GL_1(R)\;,$$
and so, for a topological space $X$, we may form the twisted cohomology $R^*(X; H)$ associated to a class $[H] \in H^{n+2}(X)$.  Furthermore, since $\iota$ is invertible in $\KS_n$, the element $\phi(\iota)$ must be invertible in $R$.

\begin{defn} \label{period_defn}

We will call the data of the $A_\infty$ map $\phi: \KS_n \to R$ a \emph{periodic twisting of $R$ by $n$-gerbes}.  An equivalent (if less colorful) term is: $R$ is an 
{\it $\KS_n$-algebra spectrum}.

\end{defn}

The spectrum $R = \KS_n$ obviously admits the universal periodic twisting by $n$-gerbes.  Theorem \ref{non_nilp_thm} ensures that $\KA_n$ is another nontrivial example.  An immediate consequence of being an algebra spectrum over a rational algebra is the following:

\begin{prop} If $R$ admits a periodic twisting by $n$-gerbes for $n > 1$, then it is rational. \label{rat_prop} \end{prop}

\section{Chern characters, $\KA$-orientations, and Umkehr maps}
\label{Sec Chern}

Many approaches to T-duality involve cohomological expressions. In order to deduce T-duality isomorphisms in ordinary cohomology,
we investigate Chern characters and orientations associated with the spectra $\ka_n$ and $\KA_n$ introduced in \S\ref{sec:iterated_twistings_ku}. 
These are, of course, also interesting in their own right. 

\subsection{The Chern-Dold character}
\label{Sec Dold} 

The Chern character in complex topological K-theory of a manifold $X$ is a ring homomorphism 
from $KU^0(X)$ to even rational cohomology $H^{\rm even}(X; \Q)$
and from $KU^1(X)$ to odd rational cohomology $H^{\rm odd}(X; \Q)$. 
This periodicity can be encoded by
the homomorphism of $\Z$-graded rings 
$KU^{*} (X) \to H^*(X; \Q[u^{\pm 1}])$, 
where the latter is rational cohomology with 
coefficients in $KU^*({\rm pt}) \otimes \Q \cong \Q[u^{\pm 1}]$,
and $u^{-1}$ is the Bott generator. 
The Chern character may be regarded as a map of spectra 
$$\ch: KU \to H\Q[u^{\pm 1}],$$
where $H\Q[u^{\pm 1}]$ is 2-periodic rational cohomology.
The target is a priori formal power series over $\Q$, but becomes polynomials when 
evaluating on finite-dimensional manifolds.

\medskip
The above Chern character has a generalization to any generalized 
cohomology theory. 
The standard Chern-Dold character (see, e.g., ch. 14 of \cite{Dold}, by Dold) is based upon the Eilenberg-MacLane spectrum for the rationalization of the coefficients of our theory. Let $\ka_n^*({\rm pt})=\pi_*(\ka_n) = R^*=\sum_j R^j$.  Unfortunately, we have little understanding of an explicit formula for these coefficients.  Further, the results of \cite{ausoni} suggest that $R^*$ is most likely a very complicated ring even for $n=2$.  Nonetheless, we may make the following construction:

\begin{defn} 
The {\it Chern-Dold character for the theory} $\mathfrak{a}_n$ is the map of
cohomology theories
$$
\ch_{\mathfrak{a}_n}: \mathfrak{a}_n^* \to H^*(-; R^* \otimes \Q)
$$
induced by the rationalization map
$$\ch_{\ka_n}: \ka_n \to (\ka_n)_\Q \simeq H(R_*\otimes \Q),$$
where we identify the rationalization of $\ka_n$ with the generalized Eilenberg-MacLane spectrum associated to the graded ring $R_*\otimes \Q$.

\end{defn}

The following properties of the above character are evident from the definition and follow similarly 
to those of the standard Chern-Dold character (see \cite{Bu, Dold}). 

\begin{lem} (Properties of the Chern-Dold character)

\begin{enumerate}\itemsep0em
\item Over a point, $\ch_{\mathfrak{a}_n}$ is the canonical homomorphism $R^* \to R^* \otimes \Q$.
\item For a finite complex $X$ the homomorphism 
$$
\ch_{\mathfrak{a}_n} \otimes \Q: \mathfrak{a}_n^*(X) \otimes \Q \to H^*(X; R^* \otimes \Q)
$$
is an isomorphism.
\item $\ch_{\mathfrak{a}_n}$ is a ring homomorphism. 
\end{enumerate}
\end{lem}

\subsection{A (twisted) higher Chern character via $THH$}
\label{Sec higherChern}

Although we do not have a good understanding of the ring $R^*$ or its rationalization, there is a stand-in for the Chern-Dold character with target a recognizable generalized Eilenberg-MacLane spectrum.

\medskip
Let $n$ be odd, and consider the composite $\ch_n^{\geq 0}:$
$$\xymatrix@1{\ka_n = K^{(n-1)}(ku) \ar[drrr]_-{\ch_n^{\geq 0}} \ar[r]^-{tr^{n-1}} & THH^{(n-1)}(ku) \ar[r] & THH^{(n-1)}(ku_\Q) & THH^{(n-1)}((\Sigma^\infty \C P^\infty_+)_\Q) \ar[l]_-\simeq \ar[d]^-{ev} \\
 & & & (\Sigma^\infty K(\Z, n+1)_+)_{\Q}}$$
\noindent described in the proof of Theorem \ref{non_nilp_thm}; here $ev$ is the basepoint evaluation projection from $THH^{(n-1)}(\Sigma^\infty \C P^\infty_+) = \Sigma^\infty \Map(T^{n-1}, K(\Z, n+1))_+$ to $\Sigma^\infty K(\Z, n+1)_+$.  The target of this map is the generalized rational Eilenberg-MacLane ring spectrum $(\Sigma^\infty K(\Z, n+1)_+)_{\Q} \simeq H\Q[\beta_n]$ whose homotopy is the ring $\Q[\beta_n]$.

\begin{defn}

For $n$ odd, define the \emph{higher Chern character} $\ch_n: \KA_n \to H\Q[\beta_n^{\pm 1}]$ as the localization (at $\beta_n$) of the map $\ch_n^{\geq 0} \colon \ka_n \to H\Q[\beta_n]$.

\end{defn}

We note that when $n=1$, this is the usual Chern character, which may be alternatively described as the map $KU \to KU_{\Q} \simeq (\Sigma^\infty \C P^\infty_+[\beta^{\pm 1}])_{\Q}$ induced by rationalization and Snaith's (or Segal's) theorem.  More generally, $\ch_n$ factors the Chern-Dold character and evidently forgets some information regarding the rational homotopy of $\KA_n=\ka_n[\beta_n^{-1}]$:
$$\xymatrix{
\ka_n \ar[r]^-{\ch_{\ka_n}} \ar[d] & H(R_* \otimes \Q) \ar[d] \\
\KA_n \ar[r]_-{\ch_n} & H\Q[\beta_n^{\pm 1}]}$$
(the right vertical map is induced by $\ch_n^{\geq 0}$ in rational homotopy).  Nonetheless, it is useful on account of the fact that the coefficients of the codomain of $\ch_n$ are computed.  Furthermore, this map is highly structured:

\begin{prop}

The higher Chern character $\ch_n: \KA_n \to H\Q[\beta_n^{\pm 1}]$ is a map of $E_\infty$ ring spectra.  Further, the composite $\KS_n \to \KA_n \to H\Q[\beta_n^{\pm 1}]$ is an equivalence, and splits $H\Q[\beta_n^{\pm 1}]$ off of $\KA_n$.

\end{prop}

\begin{proof}
The work of \cite{BGT2} ensures that the Dennis trace is an $E_\infty$ map.  The original map\footnote{which need not agree with the connective cover of the map $\ch_n$!} $\ch_n^{\geq 0}: \ka_n \to  (\Sigma^\infty K(\Z, n+1)_+)_{\Q}$ in the definition of $\ch_n$ may then be seen to be $E_\infty$: it is a composite of iterates of $tr$, rationalizations, an equivalence induced by a map of ring spectra, and the map $ev$, which is easily seen to be $E_\infty$.  Finally $\ch_n$ is also $E_\infty$, being obtained from this by localization at $\beta_n$.

To obtain the second statement, it suffices to note that $\iota \in \pi_{n+1} \KA_n$ is carried to $\beta_n$, which was shown in the proof of Theorem \ref{non_nilp_thm}.

\end{proof}

Consequently $\ch_n$ induces a map of spectra $\ch_n: \gl_1(\KA_n) \to \gl_1(H\Q[\beta_n^{\pm 1}])$.  Further, since the composite 
$$\xymatrix@1{\Sigma^\infty K(\Z, n+1)_+ \ar[r]^-{t_n} & \ka_n \ar[r]^-{\ch_n^{\geq 0}} & (\Sigma^\infty K(\Z, n+1)_+)_{\Q}}$$
is the rationalization map, it follows that $\ch_n$ carries the twisting $t_n$ of $\KA_n$ by $H^{n+2}$ to the standard\footnote{Traditionally when considering twisted periodic cohomology, the period of the theory is 2, but of course higher degrees also arise (see \cite{S2}).
In this setting, the period is somewhat longer, since $\beta_n$ has dimension $n+1$.  Since $n$ is odd, the period is still even, so the two are naturally comparable.} twisting of periodic cohomology by $H^{n+2}$. 

\begin{cor}

For each $H \in H^{n+2}(X)$, the higher Chern character $\ch_n$ is a natural transformation of twisted cohomology theories:  
$$\ch_n: \KA_n^*(X; H) \to H\Q^*(X; H)[\beta_n^{\pm 1}]$$
which splits the target off of the source.

\end{cor}

\subsection{$\KA$-orientations and Umkehr maps}
\label{Sec Orient}

In order to construct the T-duality isomorphism in twisted $\KA_n$-theory via a Fourier-Mukai push-pull formula, we will need to have an Umkehr map in $\KA_n$-theory.  Here we set out a framework for constructing and analyzing the Umkehr map associated to a sphere bundle.  We will work with an arbitrary ring spectrum $R$ equipped with a periodic twisting of $n$-gerbes $\phi: \KS_n \to R$, but we have $R = \KA_n$ in mind throughout.

\medskip
Let $\pi \colon E \to X$ be a smooth fiber bundle with fiber the $q$-sphere $S^q$, and let $V \to E$ denote the vertical tangent bundle of the fiber bundle $\pi$.  
Consider the the unit map $\eta: {S} \to R$ from the sphere spectrum,
and the $J$-homomorphism $J: O \to {\rm GL}_1({S})$ to the 
space of units of ${S}$, with $BJ$ its delooping. 
Choose a map $V \colon E \arr BO$ that classifies the vertical tangent bundle and let
\[
\tau_{E} \colon \xymatrix@1{E \ar[r]^-{V} & BO \ar[r]^-{BJ} & B\GL_1({S}) \ar[rr]^-{B\GL_1(\eta)} & & B\GL_1(R)}
\]
be the twist of $R$-theory that it determines.  The structure of an $R$-orientation of $V$, i.e. a Thom class in $R$-theory for the Thom spectrum $E^{V}$, is equivalent to the datum of the homotopy class of a nullhomotopy of the map $\tau_{E}$.  
\begin{defn} 
We call $\tau_{E}$ {\it the orientation twist of} $E$, 
since it is the obstruction to the $R$-orientability of the vertical tangent bundle $V$.  
\end{defn}

\medskip
Associated to the sphere bundle $\pi \colon E \to X$ is a Pontrjagin-Thom collapse map $\pi^{!}: \Sigma^{\infty} X_{+} \to E^{-V}$.  Applying $R$-cohomology, we get a twisted umkehr map $\pi_!: R^*(E^{-V}) \to R^{*}(X)$.  Since the Thom spectrum associated to the twist $-\tau_{E}$ is the $R$-module spectrum $R \sma E^{-V}$, we may rewrite the twisted umkehr map as 
\[
\pi_{!} \colon R^*(E; \tau_{E}) \to R^*(X).
\]
More generally, if $\omega \colon X \to B \GL_1(R)$ is a twist of $R$ over $X$, then there is a product twist $\tau_{E} \otimes \pi^* \omega$ induced by the action of $B \GL_1 S$ on $B \GL_1 R$ coming from the $S$-algebra structure of $R$.  The Pontrjagin-Thom collapse map induces a twisted Umkehr map of the form
\[
\pi_{!} \colon R^*(E; \tau_{E} \otimes \pi^* \omega) \to R^*(X; \omega).
\]
In the presence of an orientation, we can recover the usual Umkehr map along $\pi$.

\begin{prop} 
A trivialization of $\tau_{E}$, or equivalently an $R$-orientation of $V$, determines an (untwisted) Umkehr map
\[
\pi_{!} \colon R^{*}(E) \cong R^{* - q}(E; \tau_{E}) \to R^{* - q}(X),
\]
defined as the composite of the resulting Thom isomorphism in $R$-theory for $E^{-V}$ with the twisted Umkehr map. 
\end{prop}

\begin{rem}  
If $R$ admits a periodic twisting by $n$-gerbes for $n = 1$, then $V$ is $R$-orientable if and only if it admits a $\spinc$ structure.  In the case of odd $n>1$, the spectrum $R$ is an $H\Q$-algebra (Prop. \ref{rat_prop}) and the virtual tangent bundle $V$ is $R$-orientable if and only if it is an orientable vector bundle in the usual sense.
\end{rem}

The twisted Umkehr map $\pi_{!}$ is natural in the variable $X$, meaning that if 
\[
\xymatrix{
E' \ar[r]^{\widetilde{f}} \ar[d]_{\pi'} & E \ar[d]^{\pi} \\
X' \ar[r]^{f} & X
}
\]
is a pullback diagram of bundles, then there is a commutative diagram
\begin{equation}\label{compatible_umkehr}
\xymatrix{
R^*(E'; \tau_{E'} \otimes \pi'^*f^*\omega) \ar[r]^-{\pi'_{!}} \ar[d]_{\widetilde{f}^*} & R^*(X'; f^* \omega) \ar[d]^{f^*} \\
R^*(E; \tau_{E} \otimes \pi^* \omega) \ar[r]^-{\pi_{!}} & R^*(X; \omega)
}
\end{equation}
relating the Umkehr maps along $\pi$ and $\pi'$ with the contravariant functoriality of the twisted cohomology theory $R^*(-; \omega)$.  Notice that we use the canonical equivalence of twists 
\[
\widetilde{f}^*(\tau_{E} \otimes \pi^* \omega) \simeq \widetilde{f}^*\tau_{E} \otimes \widetilde{f}^* \pi^* \omega \simeq \tau_{E'} \otimes \pi'^* f^* \omega
\]
to identify the target of $\widetilde{f}^*$ with the source of $\pi_{!}$.

\medskip
We will also need the following lemma on Mayer-Vietoris sequences
\begin{lem}\label{lemma:MV_sequences}
The Umkehr map $\pi_{!}$ is natural for the boundary operator $\delta$ in Mayer-Vietoris sequences for twisted $R$-theory, i.e. if $X = U \cup V$ is a decomposition of $X$ into open subsets, and we write $E_{U \cap V}$ for the restriction of the bundle $E$ to the open subset $U \cap V$, then the following diagram commutes
\[
\xymatrix{
R^*(E_{U \cap V}; \tau_{E_{U \cap V}} \otimes \pi^*\omega) \ar[d]_{\delta} \ar[r]^-{{\pi \mid_{U \cap V}}_{!}} & R^*(U \cap V; \omega) \ar[d]^{\delta} \\
R^{* + 1}(E; \tau_{E} \otimes \pi^*\omega) \ar[r]^-{\pi_{!}} & R^{* + 1}(X; \omega).
}
\]
\end{lem}
\begin{proof}
We first construct the Mayer-Vietoris boundary operator $\delta$ in $R$-theory twisted by $\omega$.  We write $X^{-\omega}$ and $U^{-\omega}$, for the $R$-module Thom spectra associated to the inverse of the twist $\omega \colon X \to B \GL_1(R)$ and to the restriction of $\omega$ to $U$, where the latter is a slight abuse of notation.    The functoriality of the Thom spectrum functor gives the commutative square of spectra on the left, which we extend to a morphism of cofiber sequences
\[
\xymatrix{
(U \cap V)^{-\omega} \ar[r] \ar[d] & U^{-\omega} \ar[d] \ar[r]^-{c_1} & C_1 \ar@{-->}[d]^{\phi} \ar[r]^-{\partial} & \Sigma (U \cap V)^{-\omega} \ar[d] \\
V^{-\omega} \ar[r] & X^{- \omega} \ar[r]^-{c_2} & C_2 \ar[r]^-{\partial} & \Sigma V^{- \omega}
}
\]
The induced morphism $\phi$ on cofibers is an equivalence by the excision property of the parametrized cohomology theory associated to the parametrized $R$-module spectrum over $X$ determined by $\omega$ \cite[\S20]{MS}.  Another point of view is that the Thom spectrum functor $(-)^{-\omega}$ preserves homotopy pushouts \cite[3.13]{ABGHR2}\footnote{The cited work uses an $\infty$-categorical model for the Thom spectrum functor and so the relevant result is expressed using the language of colimits in $\infty$-categories. It is a consequence of their work and the identification of homotopy colimits with $\infty$-colimits that any model for the Thom spectrum functor preserves homotopy colimits.}, such as $X = U \cup_{U \cap V} V$, and thus the induced map on cofibers is an equivalence.  The Mayer-Vietoris sequence boundary operator $\delta$ is defined to be the composite $\delta = c_2^* \circ (\phi^{-1})^* \circ \partial^*$ of the induced maps in $R$-cohomology, where $\phi^{-1}$ is a homotopy inverse to the equivalence $\phi$.  

There is a similar diagram defining the Mayer-Vietoris boundary operator associated to the decomposition $E = E_{U} \cup_{E_{U \cap V}} E_{V}$ and the naturality of the Pontrjagin-Thom construction constructs a commutative square of cofiber sequences relating the displayed diagram to the one for $E$.  This gives the desired relation between $\delta$ and the twisted Umkehr map $\pi_{!}$.
\end{proof}

\section{T-duality for sphere bundles in theories twisted by $n$-gerbes} 
\label{Sec T}

In this section, we define T-duality for sphere bundles and prove the T-duality isomorphism for cohomology theories admitting periodic twisting by $n$-gerbes.  The material closely follows that of \cite{BS} in the setting of T-duality for circle bundles.  We will also prove that when $n>1$, any cohomology theory for which such an isomorphism holds must be rational.

\subsection{Formulation of T-duality for sphere bundles}

Let $X$ be a topological space, and $\pi:E \to X$ and $\pihat: \Ehat \to X$ be $S^q$-bundles over $X$.  For simplicity, we work with smooth fiber bundles, but the arguments can be adapted to the case of topological fiber bundles.  Note that we do not require the sphere bundles to be orientable.  In the analysis of classifying spaces in \S\ref{Sec pairs}, we will work with a chosen structure group for our bundles, but for now that choice is not revelant.  

Define $E *_X \Ehat$ to be the \emph{fiberwise join} of $E$ and $\Ehat$ over $X$.  Note that this is a bundle with fiber the join $S^q * S^q = S^{2q+1}$, and that there are natural fiberwise embeddings
$$\begin{array}{ccc}
i: E \hookrightarrow E *_X \Ehat & {\rm and} & \ihat: \Ehat \hookrightarrow E *_X \Ehat
\end{array}$$
given by inclusion of each factor in the join.

\begin{defn}

A \emph{bundle Thom class}\footnote{This is simply called a \emph{Thom class} in \cite{BS}; this usage conflicts with the usual notion of a Thom class as a generator $\theta$ for the cohomology of a Thom space.  The two notions are closely related, since $\theta$ is carried to $\Th$ under the collapse map from a sectioned sphere bundle to its Thom space.} for a $S^m$-bundle $p: Y \to X$ is an element $\Th \in H^{m}(Y; \Z)$ with the property that its restriction to each fiber is a generator of $H^{m} (S^{m}; \Z)$.  Equivalently, $p_!(\Th) = \pm1 \in H^0(X; \Z)$.

\end{defn} 

We let $n=2q-1$, which is evidently odd, and consider representatives 
$$H \colon E \to K(\Z, n + 2) \quad \mbox{ and } \quad \Hhat \colon \Ehat \to K(\Z, n + 2)$$ 
for cohomology classes
$$[H] \in H^{n+2}(E; \Z) \quad \mbox{ and } \quad [\Hhat] \in H^{n+2}(\Ehat; \Z).$$

\begin{defn} \label{t-dual_defn}

We say that the pairs $(E, H)$ and $(\Ehat, \Hhat)$ are \emph{T-dual} if there exists a bundle Thom class $\Th \in H^{n+2}(E *_X \Ehat)$ with $i^* \Th = [H]$ and $\ihat^* \Th = [\Hhat]$. 

\end{defn}

Consider the \emph{correspondence space} $E \times_X \Ehat$, which is an $S^q \times S^q$-bundle over $X$:
$$\xymatrix{
 & E \times_X \Ehat \ar[dl]_-{p} \ar[dr]^-{\phat} & \\
E \ar[dr]_-{\pi} & & \Ehat\;. \ar[dl]^-{\pihat} \\ 
 & X &
}$$
There is a tautological homotopy $h: I \times (E \times_X \Ehat) \to E *_X \Ehat$ from $i \circ p$ to $\ihat \circ \phat$ that is given by the formula $h_t(e, \ehat) = (t, e, \ehat)$ and recognizes the fiberwise join $E *_X \Ehat$ as a quotient of $I \times (E \times_X \Ehat)$.

\medskip
Pulling the twisting classes back over $p$ and $\phat$ gives cohomology classes $p^* [H]$ and $\phat^* [\Hhat]$,
respectively, on the correspondence space.  But since $i^* \Th = [H]$ and $\ihat^* \Th = [\Hhat]$, we have a homotopy between the maps representing $p^* [H]$ and $\phat^* [\Hhat]$:
$$\Lambda = \Th \circ h: H \circ p = \Th \circ i \circ p \to  \Th \circ \ihat \circ \phat =  \Hhat \circ \phat.$$
The T-duality transformation $T$ is the composite
\[
\xymatrix{
R^*(E; \tau_{E} \otimes H) \ar[r]^-{p^*} & R^*(E \times_{X} \Ehat; p^*\tau_{E} \otimes p^* H) \ar[d]^-{\Lambda}_-{\cong} & \\
& R^*(E \times_{X} \Ehat; \tau_{E \times_{X} \Ehat} \otimes \phat^*\Hhat) \ar[r]^-{\phat_{!}} &R^*(\Ehat; \Hhat),
}
\]
where the middle isomorphism is induced by the homotopy $\Lambda$ and the canonical isomorphism of the pullback along $p$ of the vertical tangent bundle of $\pi \colon E \to X$ with the vertical tangent bundle of $\phat \colon E \times_{X} \Ehat \to \Ehat$.  We may now state our main theorem.

\begin{thm} \label{t_dual_thm}

Let $R$ be an $A_\infty$ ring spectrum, and let $\phi: \KS_n \to R$ be a periodic twisting of $R$ by $n$-gerbes.  Suppose that $X$ is homotopy equivalent to a finite CW complex, that $(E, H)$ and $(\Ehat, \Hhat)$ are T-dual, and write $\tau_E: X \to B\GL_1(R)$ for the orientation twist of the vertical tangent bundle of $E$.  Then the T-duality transformation 
$$T:= \phat_! \circ \Lambda \circ p^*: R^*(E; \tau_E \otimes H) \to R^{*}(\Ehat; \Hhat)$$
is an isomorphism.

\end{thm}

Recall that our main examples of such an $R$ are $\KS_n$ itself (Definition \ref{def S}), 
and $\KA_n$ (Definition \ref{def U}).  

\subsection{Proof of Theorem \ref{t_dual_thm}}

We will establish Theorem \ref{t_dual_thm} through a series of lemmata. 

\begin{lem}
\label{lemma 1}
Theorem \ref{t_dual_thm} holds when $X$ is a point.

\end{lem}

\begin{proof}

In this case, the Atiyah-Hirzebruch spectral sequence for the $R$-cohomology of $E \times \Ehat = S^q \times S^q$ contains $R^*(E)$ and $R^*(\Ehat)$ as direct summands on every page, and thus cannot have any non-zero differentials.  This gives an isomorphism of rings
$$R^*(E \times \Ehat) \cong R^*[x, \xhat]/(x^2, \xhat^2),$$
where $x$ and $\xhat$ are represented by singular cohomology classes of dimension $q$.  Write 
$\beta = \phi_*(\iota) \in \pi_{n+1}(R)$ for the periodicity class; when $R = \KA_n$, $\beta = \beta_n$.  We claim that $\Lambda$ is given by multiplication by the degree 0 class $1\pm \beta x \xhat$, where we are regarding $\beta$ as a $-2q = -(n+1)$-dimensional cohomology class via the identification $R^*(\mathrm{pt}) = \pi_{-*} R$.

Granting that, the result follows: $\phat_! \circ \Lambda \circ p^*(1) = \phat_!(1\pm \beta x \xhat) = \pm \beta \xhat$, and 
$$\phat_! \circ \Lambda \circ p^*(x) = \phat_![(1\pm \beta x \xhat)x] = \phat_!(x) = 1$$
since the Umkehr map $\phat_!$ is given by division by the class $x$ if possible, and is 0 otherwise. 

To see the claim, we first recall that for a trivial cohomology class $\alpha = 0 \in H^m(Y; \Z)$, represented by a map $\alpha: Y \to K(\Z, m)$, the set of homotopy classes of nullhomotopies of $\alpha$ is a torsor for 
$$[Y, \Omega K(\Z, m)] = H^{m-1}(Y; \Z) = H^m(\Sigma Y, \Z),$$
by concatenation of homotopies.  In particular, it is in bijection with this group.

We note that since $H^{n+2}(E) = 0 = H^{n+2}(\Ehat)$, we have $H = \Hhat = 0$.  So $p^*H$ and $\phat^* \Hhat$ are necessarily null.  Further, since $H^{n+1}(E) = 0 = H^{n+1}(\Ehat)$, there is a unique homotopy class of trivialization (nullhomotopy) of $H$ and $\Hhat$, yielding preferred nullhomotopies of $p^* H$ and $\phat^* \Hhat$.  Note that the nullhomotopies of the pulled back classes are not unique, since the group $H^{n+1}(E\times \Ehat) =\Z$ affords many possibilities.  

Indeed, the homotopy $\Lambda$ of Theorem \ref{t_dual_thm} is a nontrivial homotopy between the trivialized twists $p^* H$ and $\phat^* \Hhat$.  We may therefore regard it as being represented through the action of an element of the group $H^{n+1}(E \times \Ehat) \cong H^{n+2}\Sigma (E \times \Ehat)$ on the trivial twist.  But by construction, $\Lambda$ is given as a composite
$$\xymatrix@1{\Sigma(E \times \Ehat) \ar[r] & E * \Ehat = S^{n+2} \ar[r]^-{\Th} & K(\Z, n+2).}$$
Here the first map arises from the tautological homotopy $I \times E \times \Ehat \to E * \Ehat$; the factorization through $\Sigma(E \times \Ehat)$ uses the fact that the inclusions of $E$ and $\Ehat$ into $E * \Ehat$ are nullhomotopic.  In particular, this map induces an isomorphism in $H^{n+2}$.  By assumption, this composite is a generator of $H^{n+2}\Sigma (E \times \Ehat)$.  The map that $\Lambda$ induces in $R^*$ is therefore given by multiplication by the invertible class
$$
\xymatrix@1{E \times \Ehat \ar[rr] \ar@{.>}[dr] & & \Omega S^{n+2} \ar[rr]^-{\Omega \Th} && K(\Z, n+1) \ar[rr]^-{\GL_1 \phi} && \GL_1 R \\
 & S^{n+1} \ar@{.>}[ur] \ar@{.>}[urrr]^-{\pm \iota} \ar@{.>}[urrrrr]_-{1 \pm \beta} & & & & & }.
 $$
The result follows, since the top-dimensional cohomology class of $S^{n+1}$ pulls back to $x \xhat$ on $E \times \Ehat$.

\end{proof}

We now establish the naturality properties of the T-duality transformation that we need to prove Theorem \ref{t_dual_thm} by cellular induction.

\begin{lem} 
\label{lemma 2}
Let $f:Y \to X$ be a continuous map, and define $F: f^*E \to E$ and $\Fhat: f^* \Ehat \to \Ehat$ to be induced maps from the pullbacks of $E$ and $\Ehat$ along $f$.  Then the T-duality map for the bundles $E$ and $\Ehat$ on $X$ pulls back over $f$; that is, the following diagram commutes:
$$\xymatrix{
R^*(E, \tau_E \otimes H) \ar[r]^-T \ar[d]_-{F^*} & R^{*}(\Ehat, \Hhat) \ar[d]^-{\Fhat^*}\\
R^*(f^*E, F^*(\tau_E \otimes H)) \ar[r]^-T & R^{*}(f^*\Ehat, \Fhat^*\Hhat).
}$$

\end{lem}
\begin{proof}
The maps $p^*$ and $\Lambda$ commute with $F^*$ and $\Fhat^*$  by the contravariant functoriality of twisted cohomology theories.  The twisted Umkehr map $\phat_{!}$ also commutes with $F^*$ and $\Fhat^*$ by diagram \eqref{compatible_umkehr}.
\end{proof}

\begin{lem} 
\label{lemma 3}
If $X$ is decomposed as the union $X = U \cup V$ of open subsets, we may restrict $E$ and $\Ehat$ to these subsets, obtaining $E = E_U \cup E_V$ and $\Ehat = \Ehat_U \cup \Ehat_V$.  Further, let $\delta$ and $\deltahat$ be the boundary operators in the Mayer-Vietoris sequences for these decompositions of $E$ and $\Ehat$, 
respectively.  Then the T-duality transformation preserves these Mayer-Vietoris sequences; that is, %
$$T \circ \delta = \deltahat \circ T: R^*(E_{U \cap V}, (\tau_E \otimes H)|_{E_{U \cap V}}) \to R^{* + 1}(\Ehat, \Hhat).$$

\end{lem}
\begin{proof}
We argue that each component of $T = \phat_{!} \circ \Lambda \circ p^*$ commutes with the relevant boundary operator.
The pullback map $p^*$ is induced by a map of spectra over $E$, hence extends to a natural transformation of parametrized cohomology theories on spaces over $E$ \cite[\S20]{MS}.  The data of the boundary map $\partial$ in the cofiber sequence associated to a pair of spaces over $E$, such as in a Mayer-Vietoris sequence, is preserved by such a natural transformation.  Similarly, the map $\Lambda$ is induced by a map of spectra over $E \times_{X} \Ehat$, so commutes with the boundary operators.  The twisted Umkehr map $\phat_{!}$ commutes with the boundary operators by Lemma \ref{lemma:MV_sequences}.

\end{proof}

The proof of Theorem \ref{t_dual_thm} immediately follows by cellular induction, using the previous three results,
Lemmas \ref{lemma 1}--\ref{lemma 3}.

 \begin{rem}
Note that if $R'$ is an $A_\infty$ ring spectrum and $\psi: R \to R'$ an $A_\infty$ map, the composite $\psi \circ \phi$ is a periodic twisting of $R'$ by $n$-gerbes, and so the results of this theorem also hold for $R'$.  Further, it is evident from the proof that $\psi$ throws the T-duality isomorphism for $R$ onto that for $R'$.  
\end{rem}

Summarizing this result in the case that $R = \KA_n$, $R' = H\Q[\beta_n^{\pm 1}]$, and $\psi = \ch_n$, we have:

\begin{cor}\label{chern_corollary}

The higher Chern character $\ch_n$ throws the T-duality isomorphism for $\KA_n$ onto that of periodic rational cohomology, in the sense that the following diagram commutes:
$$\xymatrix@1{
\KA_n^*(E; \tau_E \otimes H) \ar[rr]^-{T_{\KA_n}}_-{\cong} \ar[d]_-{\ch_n} & & \KA_n^{*}(\Ehat; \Hhat) \ar[d]^-{\ch_n} \\
H\Q^*(E; \tau_E \otimes H)[\beta_n^{\pm 1}] \ar[rr]^-{\cong}_-{T_{H\Q[\beta_n^{\pm 1}]}} & & H\Q^{*}(\Ehat; \Hhat) [\beta_n^{\pm 1}]\;. 
}$$

\end{cor}

\subsection{A T-duality isomorphism implies rationality}

Let $R$ be an $A_\infty$ ring spectrum which admits a twisting by $n+2$-dimensional cohomology classes.  That is, we take as given a map
$$\tau: K(\Z, n+2) \to B\GL_1(R)$$
from which we may define twisted cohomology groups $R^*(X, H)$ for $H \in H^{n+2}(X,\Z)$.  Note that we do \emph{not} assume that this twisting is periodic in the sense of Definition \ref{period_defn}.  In other words, the map 
\[
t \colon \Sigma^{\infty}_{+} K(\Z, n + 1) \longrightarrow R
\]
of ring spectra adjoint to $\tau$ is not assumed to factor through the Arthan spectrum $\KS_n =  \Sigma^{\infty}_{+} K(\Z, n + 1)[\iota^{-1}]$.  In particular, we may not immediately conclude from Proposition \ref{rat_prop} that $R$ is a rational spectrum.

One may ask whether a T-duality isomorphism holds in this setting.  That is, if $(E, H)$ and $(\Ehat, \Hhat)$ are T-dual $S^q$-bundles\footnote{We still require $n=2q+1$.} over $X$ in the sense of Definition \ref{t-dual_defn}, we may, as before, form the T-duality transformation $T$ as the composite
\[
\xymatrix{
R^*(E; \tau_{E} \otimes H) \ar[r]^-{p^*} & R^*(E \times_{X} \Ehat; p^*\tau_{E} \otimes p^* H) \ar[d]^-{\Lambda}_-{\cong} & \\
& R^*(E \times_{X} \Ehat; \tau_{E \times_{X} \Ehat} \otimes \phat^*\Hhat) \ar[r]^-{\phat_{!}} &R^*(\Ehat; \Hhat),
}
\]
and ask whether $T$ is an isomorphism.

\begin{thm} \label{rational_thm}

The following are equivalent: 
\begin{enumerate}
\item The T-duality transformation $T$ is an isomorphism for all base spaces $X$ of the homotopy type of a finite CW complex.
\item The T-duality transformation $T$ is an isomorphism for $X$ a point.
\item The twisting $\tau$ is periodic.
\end{enumerate}
In particular, any cohomology theory which admits a T-duality isomorphism for $S^q$-bundles with $q > 1$ must be a rational cohomology theory.

\end{thm}

\begin{proof}

Clearly the first statement implies the second.  Conversely, the proof of Theorem \ref{t_dual_thm} showed that the general case reduced to the case $X$ a point using naturality (Lemma \ref{lemma 2}) and the Mayer-Vietoris sequence (Lemma \ref{lemma 3}) -- note that these results do not require the twisting to be periodic.  We have shown that the T-duality isomorphism holds when $X$ is a point under the assumption of a periodic twisting (Lemma \ref{lemma 1}), so the third statement implies the second.  So we need only show that the second statement implies the third.

Let us define $\beta: S^{n+1} \to R$ as we have done in section \ref{sec:iterated_twistings_ku}, via the composite
$$\xymatrix{S^{n+1} \ar[r]^-{\iota} & \Sigma^\infty K(\Z, n+1)_+ \ar[r]^-t & R}$$
where $t$ is adjoint to $\tau$.  We will show that $\beta$ is an invertible element of $\pi_*R$ under the assumption that T-duality holds over a point.  This then implies that the map of ring spectra $t$ factors through the Arthan spectrum $\KS_n$, and in particular makes $R$ an algebra spectrum over $\KS_n$.

As in the proof of Lemma \ref{lemma 1}, we may write $R^*(E \times \Ehat) = R^*[x, \xhat]/(x^2, \xhat^2)$ and compute the T-duality map to be the $R_*$-linear operation given by
$$1 \mapsto \pm \beta \xhat\; \mbox{ and } \; x \mapsto 1.$$
This can only be an isomorphism if $\beta$ is invertible.

\end{proof}

\section{Classifying spaces for T-dual pairs}
\label{Sec pairs} 

We continue to insist that $n=2q-1$ be odd.  In this section, we introduce spaces $R_n(G)$ and $P_n(G)$ that classify the objects considered in the previous section, at least in the case where the bundles $E$ and $\Ehat$ are oriented.  More precisely, let $G$ be a topological group equipped with a homomorphism to $\Homeo^+(S^q)$, the group of orientation-preserving homeomorphisms of the $q$-sphere $S^q$.  

We will construct natural isomorphisms that give the following identifications:
\begin{itemize}
\item $[X, R_n(G)]$ is the set of equivalence classes of pairs $(E, H)$, where $E \to X$ is an $S^q$-bundle with structure group $G$, and $H$ is a map representing a twisting class $[H] \in H^{n+2}(E)$.
\item $[X, P_n(G)]$ is the set of equivalence classes of triples $(E, \Ehat, \Th)$, where $E$, $\Ehat$ are $S^q$-bundles with structure group $G$, and $\Th \in H^{n+2}(E*_X \Ehat)$ is a bundle Thom class.
\end{itemize}
There is a map $P_n(G) \to R_n(G)$ that induces the natural transformation which carries $(E, \Ehat, \Th)$ to $(E, i^* \Th)$ (where we make a slight abuse of notation and write $i^*\Th$ for a choice of representing map).  We will analyze the homotopy type of the spaces involved and show that this map is an equivalence precisely when $q=1$. 

\begin{rem}
We may regard $[X, P_n(G)]$ as the set of T-dual pairs over $X$.  The failure of this map to be an equivalence for $q>1$ shows that it is not the case that every pair $(E, H)$ has a T-dual $(\Ehat, \Hhat)$, and that when a T-dual exists, there is no guarantee that it is unique.
However, we will encounter special cases towards the end of section \ref{Sec_Compare} where T-duals 
do exist and can be unique. 
\end{rem}

\subsection{Euler classes and the classifying space $P_n(G)$}
\label{Sec P}

Let $p: E \to X$ be a fiber bundle with fiber $S^m$.  If $E$ is the sub-sphere bundle $E = S(V)$ of an oriented real vector bundle $V$ over $X$ of rank $m+1$, then there is an \emph{Euler class}
$$e(E) = e(V) \in H^{m+1}(X)$$
which is the pullback of the (usual notion of the) Thom class of $X^V$ under the zero-section.

\medskip
More generally, if $E$ is an oriented $S^m$-bundle, the fiberwise unreduced suspension
$$\Sigma_X E := I \times E /\sim, \mbox{ where $(t, e) \sim (t, e')$ for $t=0,1$ when $p(e) = p(e')$}$$ 
is a topological $S^{m+1}$-bundle over $X$ which inherits an orientation from $E$.   There are two canonical sections $s_t: X \to \Sigma_X E$, where $t = 0, 1$, given by $s_t(x) = [t, e]$ for any $e$ in the fiber over $x$.  The associated Thom space 
$$X^{E\oplus 1} := \Sigma_X E / s_1(X)$$
is obtained by collapsing the image of $s_1$ to a point.  Since $\Sigma_X E$ is oriented, there is a Thom class $\theta_{E} \in H^{m+1}(X^{E\oplus 1})$ in integral cohomology with the property that the map
$$H^*(X) \to \widetilde{H}^{*+m+1}(X^{E \oplus 1}) \mbox{ given by } \alpha \mapsto (\Sigma_X p)^*(\alpha) \cup \theta_{E}$$
is an isomorphism.  We define the Euler class of $E$ by the formula $e(E) = s_0^*(\theta_{E})$. The following is a straightforward application of the Leray-Hirsch theorem and the Gysin sequence; in the case $m=1$, see section 2.2 of \cite{BS} and 3.1 of \cite{BEM}.

\begin{prop} \label{thom_prop}

The following are equivalent:
\begin{enumerate}\itemsep0em
\item $E$ admits a bundle Thom class $\Th \in H^m(E)$.
\item $H^*(E)$ is a free $H^*(X)$-module on $1, \Th$.
\item $e(E) = 0$.
\end{enumerate}

\end{prop}

The set of homotopy classes of nullhomotopies of a representative cocycle for $e(E) \in H^{m+1}(X)$, if trivial, is a torsor under $H^{m}(X)$.  The same is true of bundle Thom classes $\Th \in H^m(E)$: for any $\alpha \in H^m(X)$, $\Th + p^*(\alpha)$ is another such bundle Thom class.  This structure is indeed a torsor (as it arises from a free and transitive action of $H^m(X)$) by virtue of point 2 in the above Proposition.  Thus we may in fact conclude:

\begin{quote}
{\it The set of homotopy classes of trivializations of $e(E)$ is naturally in bijection with the set of bundle Thom classes $\Th \in H^m(E)$}. \end{quote}

The next result implies that our definition of T-duality may be described in terms of characteristic classes, as in the cases $q = 1, 3$ \cite{BEM, BS, BEM2}.

\begin{prop}  Suppose that $E$ and $\Ehat$ are orientable $S^q$-bundles over $X$, and that $(E, H)$ and $(\Ehat, \Hhat)$ are T-dual.  Then the equations
\[
e(E) = \pm \pihat_{!}[\Hhat] \quad \mbox{ and } \quad e(\Ehat) = \pm \pi_{!} [H]
\]
hold in $H^{q + 1}(X)$.  
\end{prop}
\begin{proof}
Let $\Sigma_{X}E$ and $\Sigma_{X}\Ehat$ be the oriented $S^{q + 1}$-bundles over $X$ formed by taking the fiberwise unreduced suspension.  Form the fiberwise smash product $\Sigma_{X} E \sma_{X} \Sigma_{X}\Ehat$ with respect to the canonical sections $s_1$ on either side.  This is a $S^{2q + 2}$-bundle and that inherits an orientation from $E$ and $\Ehat$.  In other words, we equip the Thom space 
\[
X^{E \oplus \Ehat \oplus 2} := (\Sigma_{X} E \sma_{X} \Sigma_{X}\Ehat) / s_1(X)
\]
with the Thom class $\theta_{E} \cdot \theta_{\Ehat} \in H^{2q + 2}(X^{E \oplus \Ehat \oplus 2})$.  Consequently, the associated Euler class satisfies
\[
e(E \ast_{X} \Ehat) = e(E) \cdot e(\Ehat).
\]
The $S^{2q + 1}$-bundle $E \ast_{X} \Ehat$ also inherits an orientation from $E$ and $\Ehat$; the Thom class on the associated Thom space agrees with $\theta_{E} \cdot \theta_{\Ehat}$ under the homeomorphism 
\[
X^{(E \ast_{X} \Ehat) \oplus 1} \cong X^{E \oplus \Ehat \oplus 2}.
\]
Using these orientations to define Thom isomorphisms and their associated Umkehr maps, the inclusion $i \colon E \arr E \ast_{X} \Ehat$ induces a commutative diagram relating the Gysin sequences for $\rho \colon E \ast_{X} \Ehat \arr X$ and $\pi \colon E \arr X$:
\[
\xymatrix{
\dotsm \ar[r] & H^{2q + 1}(X) \ar[r]^-{\rho^*} & H^{2q + 1}(E \ast_{X} \Ehat) \ar[r]^-{\rho_{!}} \ar[d]_{i^*}  & H^{0}(X) \ar[d]^{e(\Ehat)} \ar[r]^-{e(E \ast_{X} \Ehat)} & H^{2q + 2}(X) \ar[r] \ar@{=}[d] & \dotsm \\
\dotsm \ar[r] & H^{2q + 1}(X) \ar[r]_-{\pi^*} \ar@{=}[u] & H^{2q + 1}(E) \ar[r]_-{\pi_{!}} & H^{q + 1}(X) \ar[r]_-{e(E)} & H^{2q + 2}(X) \ar[r] & \dotsm
}
\]
Now suppose that $(E, H)$ and $(\Ehat, \Hhat)$ are T-dual.  In particular, we are given a bundle Thom class $\Th \in H^{2q + 1}(E \ast_{X} \Ehat)$ satisfying $i^*\Th = [H]$ and $\ihat^*\Th = [\Hhat]$.  We do not know if (the suspension of) $\Th$ is the pullback of our preferred Thom class $\theta_{E} \cdot \theta_{\Ehat}$, but we do know that $\rho_{!} \Th = \pm 1$, and consequently that $\pi_{!}[H] = \pm e(\Ehat)$.  The relation $\pihat_{!}[\Hhat] = \pm e(E)$ follows from the analogous diagram involving the Gysin sequence for $\pihat$.  

\end{proof}

We now construct the space $P_n(G)$ which parameterizes triples $(E, \Ehat, \Th)$, where $E$ and $\Ehat$ are $S^q$-bundles with structure group $G$, and $\Th$ is a bundle Thom class on $E *_X \Ehat$.

\begin{defn} 
Let $P_n(G)$ be the homotopy fiber of the composite
$$\xymatrix@1{BG \times BG \ar[r]^-{B\join} & B\Homeo^{+}(S^q * S^q) \ar[r]^-{e} & K(\Z, 2q+2) = K(\Z, n+3)}.$$
Here $\join \colon G \times G \to \Homeo^{+}(S^q * S^q)$ is the join of a pair of orientation-preserving homeomorphisms and $e$ classifies the Euler class. 
\end{defn}

By construction, homotopy classes of maps $X \to P_n(G)$ are in bijection with the set of equivalence classes of the following data:
\begin{itemize} \itemsep0em
\item a pair $(E, \Ehat)$ of $S^q$-bundles over $X$ with structure group $G$, and 
\item the homotopy class of a trivialization of the Euler class of the fiberwise join $E \ast_{X} \Ehat$.
\end{itemize}
Applying Proposition \ref{thom_prop} (or rather, its subsequent refinement) in this setting yields:

\begin{prop}\label{prop:classify_P}
 There is a bijection between $[X, P_n(G)]$ and the set of equivalence classes of triples $(E, \Ehat, \Th)$, where $\Th \in H^{n+2}(E \times_X \Ehat)$ is a bundle Thom class.  \end{prop}

\begin{rem}
Up to equivalence, the data of a triple $(E, \Ehat, \Th)$ may be identified with the data of two T-dual pairs $(E, H)$ and $(\Ehat, \Hhat)$ via $(E, \Ehat, \Th) \longmapsto (E, i^* \Th), (E, \ihat^*\Th)$.  Thus $P_n(G)$ is the classifying space for T-dual pairs.
\end{rem}

\subsection{The classifying space $R_n(G)$}
\label{Sec R}

We now work with slightly more generality.  Let $F$ be a topological space and suppose that $G$ is a topological group equipped with a continuous homomorphism $G \to \Homeo(F)$.  Consider pairs $(E, H)$ consisting of a fiber bundle $\pi \colon E \to X$ with fiber $F$ and structure group $G$, along with a map $H \colon E \arr K(\bZ, n + 2)$ representing a cohomology class $[H] \in H^{n + 2}(E; \bZ)$ on the total space that is fiberwise trivial, i.e. for every $x \in X$, the restriction of $[H]$ to $H^{n + 2}(E_{x}; \bZ)$ is zero.  We will construct a classifying space $R_{n}(G)$ for such pairs $(E, H)$.  Our construction is a generalization of Bunke-Schick's classifying space for principal circle bundles equipped with a $U(1)$-gerbe \cite{BS}.

\medskip
Let $\Map_0(F, K(\bZ, n + 2))$ denote the connected component of $\Map(F, K(\bZ, n + 2))$ containing the basepoint.  In other words, $\Map_0(F, K(\bZ, n + 2))$ is the subspace of nullhomotopic maps.  The group $G$ acts on $\Map(F, K(\bZ, n + 2))$ by $(g \cdot f)(x) = f(g^{-1} x)$.

\begin{prop}\label{prop:classify_R}
Equivalence classes of pairs $(E, H)$ are classified by the space 
\[
R_{n}(G) := EG \times_{G} \Map_0(F, K(\bZ, n + 2)).
\] 
\end{prop}

\begin{proof}
Starting with a group $G$, one can extend the $G$-Borel construction $EG \times_G(-)$ to the 
two-sided bar construction 
$B(-, G, -)$, where the first and last entries are right and left $G$-spaces, respectively.  Applying this to the space of functions from $F$ to $K = K(\bZ, n + 2)$ 
gives $R = B(\ast, G, \Map_0(F, K))$.  The total space of the universal pair over $R$ is
\[
E_{\mathrm{univ}} = B(F, G, \Map_0(F, K)) \to B(*, G, \Map_0(F, K)) = R
\]
with projection to $R$ induced by $F \to \ast$.  Note that the evaluation map $F \times \Map_0(F, K) \to K$ is $G$-invariant, hence descends to the homotopy quotient:
\[
H_{\mathrm{univ}} \colon B(F, G, \Map_0(F, K)) \to K.
\]
This map represents the universal cohomology class $[H_{\mathrm{univ}}] \in H^{n + 2}(E_{\mathrm{univ}}; \bZ)$.  The pair $(E_{\mathrm{univ}}, H_{\mathrm{univ}} )$ is the universal pair.

Given a pair $(E, H)$, we now construct a classifying map $X \to R$.  Let $P \to X$ be the principal $G$-bundle associated to $E$.  The usual contracting simplicial homotopy shows that the projection $B(G, G, P) \to P$ induced by the action of $G$ on $P$ is a homotopy equivalence.  Now form the following diagram of principal $G$-bundles
\[
\xymatrix{
P \ar[d] & B(G, G, P) \ar[d] \ar[l]_-{\simeq} \ar[r] & B(G, G, \Map_{0}(F, K)) \ar[d] \\
X \ar@{-->}[r]<-.5ex> & B(\ast, G, P) \ar[l]<-.5ex>_-{\simeq} \ar[r] & B(\ast, G, \Map_0(F, K)) = R\;.
}
\]
The map on the lower left is an equivalence by the five lemma, so we may choose a wrong-way homotopy inverse as indicated by the dashed arrow.  The maps pointed to the right are induced by the $G$-map $P \to \Map_{0}(F, K)$ adjoint to the map $H \colon P \times_{G} F = E \to K$.  The composite along the bottom gives the classifying map for the pair $(E, H)$.  

To see that the fiber bundle $E \to X$ is the pullback of the fiber bundle $E_{\mathrm{univ}} \arr R$ along this classifying map, apply $F \times_{G} (-)$ to the total spaces of the principal $G$-bundles in the diagram.

\end{proof}

\begin{rem}
More generally, if we replace $K(\bZ, n + 2)$ by any space $K$, then $R$ classifies equivalence classes of pairs $(E, H)$ consisting of a fiber bundle $E \to X$ with fiber $F$ and structure group $G$ and a map $H \colon E \to K$ that is fiberwise nullhomotopic.
\end{rem}

\subsection{Comparing $R_n(G)$ and $P_n(G)$}
\label{Sec_Compare} 

We now return to the usual setting of oriented $S^q$-bundles with structure group $G$ and set $n = 2q - 1$.  We will analyze the map $f \colon P_{n}(G) \to R_{n}(G)$ that represents the forgetful functor taking a T-dual pair $(E, H)$ and $(\Ehat, \Hhat)$ to the first item $(E, H)$.  Abbreviate $P = P_{n}(G)$ and $R = R_{n}(G)$ and note that by construction $R$ fits into a fiber sequence
\[
\xymatrix{
\Map(S^q, K(\Z, n + 2)) \ar[r] & R \ar[r] & BG.
}
\]
Also note that by the connectivity of the Eilenberg-MacLane space, $\Map(S^q, K(\Z, n + 2)$ is connected, so we are automatically restricting to nullhomotopic maps as in the construction of $R$.  

\medskip
The classifying space $P$ also participates in a fiber sequence, which we extend to the left by one entry:
\[
\dotsm \to K(\Z, n + 2) \to P \to BG \times BG \xrightarrow{ e(- \ast -)} K(\Z, n + 3).
\]
The composite of the map from $P$ with the projection ${\rm pr}_{i}$:  $BG \times BG \to BG$ to the first or second factor classifies a principal $G$-bundle with fiber $S^q$ over $P$, which we call $E$ and $\Ehat$, respectively.  The fiberwise join $E \ast \Ehat$ carries a canonical bundle Thom class by Proposition \ref{prop:classify_P}. Pulling it back along the inclusion $i_1 \colon E \to E \ast \Ehat$ gives a class $[H] \in H^{n + 2}(E)$. Choose a representing map $H \colon E \to K(\Z, n + 2)$.  By Proposition \ref{prop:classify_R}, the pair $(E, H)$ over $P$ is classified by a map $f \colon P \to R$ which fits into a map of fiber sequences
\begin{equation}\label{diag PR}
\xymatrix{
K(\Z, n + 2) \ar[d] \ar[r] & P \ar[d]^{f} \ar[r] & BG \times BG \ar[d]^{{\rm pr}_1} \\
\Map(S^q, K(\Z, n + 2)) \ar[r] & R \ar[r] & BG  \; .
}
\end{equation}

\medskip
Consider the new fibration $P \to BG$ given by passing from the upper middle to lower right of this diagram; this now has fiber $BG \times K(\Z, n+2)$.  There is a morphism of fiber sequences over $BG$:
\begin{equation}\label{compare_eqn}
\xymatrix{
BG \times K(\Z, n + 2) \ar[d] \ar[r] & P \ar[d]^{f} \ar[r] & BG  \ar@{=}[d] \\
\Map(S^q, K(\Z, n + 2)) \ar[r] & R \ar[r] & BG  \; .
}
\end{equation}
To compare the homotopy types of $P$ and $R$, we may then compare the homotopy type of their fibers.  Define a function $\te: BG \to \Map(S^q, K(\Z, n+2))$ as the composite 
$$\xymatrix@1{BG \ar[r]^-{e} & K(\Z, q+1) \simeq \Omega^q K(\Z, n+2) \ar[r] & \Map(S^q, K(\Z, n+2))}.$$
Here, $e$ is the Euler class of a $S^q$ bundle with structure group $G$, and the second map is the inclusion of the space of based maps into all maps.  Note that the fiber $\Map(S^q, K(\Z, n + 2))$ of the bottom fibration in (\ref{compare_eqn}) is an H-space, via multiplication of maps in the target.

\begin{lem}\label{fiber_lem}
In the comparison between the fibers over $P$ and $R$:
\begin{enumerate}
\item The induced map of fibers $K(\Z, n + 2) \to \Map(S^q, K(\Z, n + 2))$ in (\ref{diag PR}) is homotopic to the inclusion $x \mapsto \mathrm{const}_{x}$ of constant maps.
\item The fiber map $BG \times K(\Z, n + 2) \to \Map(S^q, K(\Z, n + 2))$ in (\ref{compare_eqn}) is homotopic to the product of the maps $\te$ and $\mathrm{const}$. \end{enumerate}
\end{lem}
\begin{proof}
The fundamental class $\iota \colon S^{n + 2} \to K(\Z, n + 2)$ represents the bundle Thom class of the trivial sphere bundle over a point.  Thus the induced map of fibers in (\ref{diag PR}) is adjoint to the composite
\[
S^q \times K(\Z, n + 2) \xrightarrow{ i_1 \times 1 } S^{2q + 1} \times K(\Z, n + 2) \xrightarrow{ \iota \times 1} K(\Z, n + 2) \times K(\Z, n + 2) \xrightarrow{+} K(\Z, n + 2).
\]
This composite is homotopic to the projection to the second factor, which proves the first claim. 

The second claim is a parameterized form of the first.  The identity map of the space $P$ classifies a triple $(E, \Ehat, \Th)$ over $P$ of $S^q$-bundles $E$ and $\Ehat$, and a bundle Thom class $\Th \in H^{n+2}(E *_P \Ehat)$.  In these terms, $f$ carries this to the map $f: P \to R$ representing the pair $(E, i_1^*(\Th))$, where $i_1: E \to E *_P \Ehat$ is the natural inclusion.

Write $X_b \cong BG \times K(\Z, n + 2)$ for the fiber over $b$, and $J_b$ for the fiber inclusion $J_b: X_b \to P$ in the top row of (\ref{compare_eqn}).  Then the composition $f \circ J_b$ represents the pair $(J_b^* E, i_1^*J_b^*(\Th))$ over $X_b$.  However, $J_b^*E$ is trivial, since its classifying map factors through $\{ b\}$.  Therefore, the fiberwise join $J_b^*(E \ast_P \Ehat) = \Sigma^{q+1}_{X_b} \Ehat$ is the $(q+1)\st$ fiberwise suspension of $\Ehat$. Then the pullback of the bundle Thom class $i_1^*J_b^*(\Th)$ is the composite
$$S^q \times BG \times K(\Z, n + 2) = S^q \times X_b \cong J_b^* E \xrightarrow{ i_1 } \Sigma^{q+1}_{X_b} \Ehat \xrightarrow{ J_b } E \ast_P \Ehat \xrightarrow{ \Th } K(\Z, n+2)\;.$$
Now, the Euler class of the $S^{n+2}$-bundle $\Sigma^{q+1}_{X_b} \Ehat$ is precisely the restriction of the bundle Thom class $J_b^* \Th$ along the zero section.  Restricting along $i_1$ gives an $S^q$-parameterized form of this fact: $i_1 \circ J_b \circ \Th: S^q \times X_b \to K(\Z, n+2)$, is the Euler class of $\Sigma^{q+1}_{X_b} \Ehat$ when restricted to any point in $S^q$.  Via adjunction, this is precisely $\te$ in the $BG$ variable.

\end{proof}

Now the following result provides special cases when T-duals exist and are unique.

\begin{cor}
\label{Cor unique}
If the Euler class map $e: BG \to K(\Z, q+1)$ is $m$-connected, so too is the comparison map $f: P \to R$.  Thus, over complexes of dimension less than $m$, there exists a unique T-dual $(\Ehat, \Hhat)$ for any $(E, H)$.  In dimension $m$, such duals exist, but are not necessarily unique.

\end{cor}

\begin{proof}

The result then follows from Lemma \ref{fiber_lem} and the fact that for an H-space such as $K(\Z, n+2)$, the product of the constant maps and based maps gives an equivalence $K(\Z, n+2) \times \Omega^q K(\Z, n+2) \simeq \Map(S^q, K(\Z, n+2))$.

\end{proof}

\begin{exmp}[Principal circle bundles]
When $n = q = 1$ and $G = U(1)$, the Euler class $e: BU(1) \to K(\Z, 2)$ is a weak homotopy equivalence; therefore the map $f$ is as well.  This recovers the result of Bunke-Schick \cite{BS} that in the case of circle bundles and $U(1)$-gerbes, T-dual pairs exist and are unique up to equivalence.
\end{exmp}

\begin{exmp}
[Principal $SU(2)$-bundles]
Suppose that $n=5$, i.e. $q = 3$, and $G = SU(2)$, and consider pairs $(E, H)$ consisting of a principal $SU(2)$-bundle $\pi \colon E \to X$ and a 5-gerbe $H \colon E \to K(\Z, 7)$ on $E$.  The Euler class $e: BSU(2) \to K(\Z, 4)$ is 5-connected, and so the same holds for $f$.  This recovers the result of Bouwknegt-Evslin-Mathai \cite{BEM2} that when the base $X$ has dim $\leq 4$, the T-dual of $(E, H)$ exists and is unique.

We can, furthermore, lift this statement to one of our newly constructed theories.
Still when $n = 5$, the spectrum $\ka_5 = K^{(4)}(ku)$ is the fourth iterated algebraic $K$-theory of $ku$ and $\KA_5 = K^{(4)}(ku)[\beta_{5}^{-1}]$ is the localization at the higher Bott element $\beta_5$ of degree 6.  The map 
\[
\tau_5 \colon K(\Z, 7) \to B \GL_1 K^{(4)}(ku)[\beta_{5}^{-1}] = B \GL_1 \KA_5
\]
defines a twisting of $\KA_5$ by 5-gerbes and we now consider the T-duality isomorphism.  Since the total space of the universal principal $SU(2)$-bundle is contractible, the vertical tangent bundle of $E$ is trivializable.  Hence the associated twist $\tau_E$ of $\KA_5$ is trivial and the T-duality isomorphism takes the form
\[
T= \phat_! \circ \Lambda \circ p^* \colon \KA_5^*(E, H) \to \KA_5^{*-3}(\Ehat, \Hhat)\;.
\]
By Corollary \ref{chern_corollary}, the Chern character throws this T-duality isomorphism onto a T-duality isomorphism 
\[
T \colon H\Q^*(E; H)[\beta_{5}^{-1}] \xrightarrow{\cong} H\Q^{* - 3}(\Ehat; \Hhat)[\beta_{5}^{-1}]
\]
in ordinary cohomology.  This recovers the T-duality isomorphism for principal $SU(2)$-bundles studied by Bouwknegt-Evslin-Mathai (Theorem 5.2 of \cite{BEM2}).
\end{exmp}

\begin{exmp}[Non-principal $SU(2)$-bundles] 

We again suppose that $n=5$, i.e. $q = 3$, and that $G = SO(4)$ acts on $S^3$ via the inclusion
\[
SO(4) = \mathrm{Isom}^{+}(SU(2)) \to \mathrm{Homeo}^{+}(SU(2)) \cong \mathrm{Homeo}^{+}(S^3).
\]
In this case, the map on fibers in the map of fiber sequences comparing $P$ and $R$ (\ref{compare_eqn}) is
\[
\widetilde{e} \cdot \mathrm{const} \colon BSO(4) \times K(\Z, 7) \to \Map(S^3, K(\Z, 7)) \simeq K(\Z, 4) \times K(\Z, 7).
\]
The effect on $\pi_4$ is the projection $+ \colon \Z \oplus \Z \to \Z$.  This means that over $S^4$, each pair $(E, H)$ of an $S^3$ bundle with structure group $SO(4)$ and a $5$-gerbe $H$ has an infinite rank one lattice of possible T-duals $(\Ehat, \Hhat)$.  In particular, spherical T-duals are not unique in this case.  This analysis recovers the construction of spherical T-duals given by Bouwknegt-Evslin-Mathai in the case of $M = S^4$  \cite{BEM3}.  They construct an infinite family of T-duals over any compact, oriented, simply-connected, 4 dimensional manifold $M$, but in this case the classification of all possible T-dual pairs is a more intricate problem.  Given a T-dual pair $(E, H)$ and $(\Ehat, \Hhat)$, we have a T-duality isomorphism in $\KA_5$-cohomology and Bott-inverted rational cohomology as in the previous example.
\end{exmp}

\begin{exmp} 
[Non-orientable bundles]
In \cite{Ba14} Baraglia extended topological T-duality, as an isomorphism in twisted K-theory,
 to the case of general circle bundles. 
Via a geometric approach, he proves existence and uniqueness of T-duals. This requires  considering 
a twist using real line bundles, i.e. arising from the factor $K(\Z/2, 1) \subset B\GL_1(K)$, in addition to the 
usual determinantal factor $K(\Z, 3)$. Specifically: the map $O(2) \to \Homeo(S^1)$ is a homotopy equivalence (this is a theorem of Kneser \cite{kneser}), so a general $S^1$-bundle fails to be equivalent to a principal bundle only on the basis of orientability.

This was put in the homotopy-theoretic 
 framework of Bunke-Schick by Mathai and Rosenberg \cite[Theorem 2.2]{MR14}, who constructed a classifying space for T-dual pairs in the not necessarily oriented setting.
When we defined the classifying space $P_n(G)$, we needed
our sphere bundles to be oriented to define the Euler class.  We may extend Mathai-Rosenberg's construction of the classifying space for T-dual pairs to the case of higher dimensional sphere bundles by considering not-necessarily oriented $S^q$-bundles $E$ and $\Ehat$ with structure group $G$ over a space $X$.  Let $\widetilde{X}$ be the double cover of $X$ which orients $E \ast_{X} \Ehat$, and $Q$ the oriented $S^{n+1}$-bundle over $\widetilde{X}$ to which $E \ast_{X} \Ehat$ pulls back.  One may show that triples $(E, \Ehat, \Th)$, where $\Th \in H^{n+2}(Q)$, are classified by a space $\widetilde{P}_n(G)$, which is defined as the homotopy fiber of the composite
$$\xymatrix@1{BG \times BG \ar[r]^-{B\join} & B\Homeo(S^q * S^q) \ar[r]^-{(w_1, e)} & E\Z / 2 \times_{\Z / 2} K(\Z, n+3).}$$
Here $\Z/2$ acts on $K(\Z, n+3)$ by negating the fundamental class.  The Borel construction fibers over $K(\Z/2, 1)$, with fiber $K(\Z, n+3)$; $(w_1, e)$ is the data of the first Stiefel-Whitney class of the $S^q * S^q$ bundle $E \ast_{X} \Ehat$ (i.e., $w_1$ mapping into $K(\Z/2, 1)$), and $e$ is the Euler class of the oriented cover $Q$.

The construction of the classifying space $R_n(G)$ for pairs $(E, H)$ is applicable in the case where $E$ is a not necessarily oriented sphere bundle with structure group $G$.  As in the oriented case, there is a comparison map $f \colon \widetilde{P}_n(G) \to R_n(G)$, and we may study its connectivity to determine the existence and uniqueness of T-dual pairs in this setting.  In the case where $n = q = 1$ and $G = O(2)$, the map $f$ is an equivalence, so we recover Baraglia's theorem on the existence and uniqueness of T-duality for non-principal circle bundles.

Our proof of the T-duality isomorphism applies to non-orientable bundles: looking back to the definition of what it means for $(E, H)$ and $(\Ehat, \Hhat)$ to be T-dual, we allow $E$ and $\Ehat$ to be arbitrary smooth bundles with fiber $S^q$.  This is enough to prove the T-duality isomorphism; indeed, the orientation twist $\tau_E$ precisely carries the correction term for non-orientable sphere bundles.

Lastly, we recall Smale's conjecture that the inclusion $O(q+1) \to \Diff(S^q)$ is a homotopy equivalence; this has been verified by Smale for $q=2$ and Hatcher for $q=3$ \cite{smale, hatcher}.  These results allow us to extend to dimensions two and three the results described above for $S^1$-bundles, yielding a similar construction and analysis of the classifying space $\widetilde{P}_n(G)$ of pairs $(E, H)$, where $E$ is a not-necessarily orientable $S^q$-bundle with structure group $G$, and $H \in H^{n+2}(E)$.

\end{exmp}

\section{Higher categories, $n$-vector spaces and iterated algebraic K-theory}
\label{Sec ncat} 

In this section we return to the question raised in the introduction on the relationship between the iterated algebraic K-theory $\ka_n = K^{(n - 1)}(ku)$, the Bott-inverted iterated algebraic $K$-theory $\KA_n = K^{(n - 1)}(ku)[\beta_{n}^{-1}]$, and
$n$-vector spaces.  The discussion is entirely speculative, and is independent of the results proved in the paper.  Our aim is to sketch a  conceptual framework which might be useful for analyzing the geometric content of T-duality for iterated algebraic $K$-theory.  Making this material rigorous would, at a minimum, require a good theory of weighted colimits for enriched $\infty$-categories.  Also, for the sake of clarity, we ignore the size issues that arise.

\medskip
We start with the concept that one should be able to associate an object like a vector bundle to a principal 
$K(\Z, n)$-bundle.  We call such an object a (rank one)
 $n$-vector bundle. Often the structure group $K(\Z, n)$ is written $B^{n-1}U(1)$ to highlight the categorical 
level: the notion of an $n$-vector bundle is an $(n - 1)$-fold categorification of the notion of a vector bundle. The association of an $n$-vector bundle to a principal $K(\Z, n)$-bundle works similarly to the classical 
theory, via an action of the group $K(\Z, n)$ on some structure. We take the point of view that the canonical object to act on is 
the linear $(n - 1)$-category $(n - 1)\Vect_{\C}$ of $(n - 1)$-vector spaces.  The iterative definition of $n$-vector space given below is arranged so that the $(n - 1)$-category $(n - 1)\Vect_{\C}$ is the basic example of a rank one $n$-vector space.

\medskip 
For small values of $n$ this is completely rigorous. When $n=1$ one has the 
canonical action of $K(\Z, 1) = U(1)$ on the complex vector space $\mathbb{C}$.  When $n=2$ 
the group $K(\Z, 2) = BU(1)$ is equivalent to the Picard groupoid of the category of vector 
spaces, i.e. to the symmetric monoidal groupoid of 1-dimensional complex vector spaces under the tensor product. As such, $BU(1)$ acts on the category 
$\Vect_{\C}$ of finite rank complex vector spaces via the tensor product. 
So to a $BU(1)$-principal bundle is canonically associated a bundle with fiber $\Vect_{\C}$.  The 2-category $2\Vect_{\C}$ of 2-vector spaces is defined to be the 2-category of finite rank $\Vect_{\C}$-modules in $\C$-linear categories, as in Kapranov-Voevodsky \cite{KV}.  Since $\Vect_{\C}$ is 
a rank one object in $2\Vect_{\C}$, the bundle associated to a principal $K(\Z, 2)$-bundle is a line 2-vector bundle. 

\medskip
The first few values of $n$ have explicit applications \cite{KV, SSS1, FHLT}.  Intuitively, the pattern continues to higher $n$, but the situation becomes
successively more complicated and less clear as $n$ grows. 

\medskip
We will sketch an approach to $n$-vector spaces using enriched higher categories.  If $\cV$ is a monoidal $\infty$-category, let us write $\iCat^{\cV}$ for the $(\infty, 1)$-category of $\cV$-enriched $\infty$-categories and $\cV$-functors.  Although our arguments are merely heuristics, much of what we say can be implemented in the model developed by Gepner-Haugseng \cite{GH}.  For example, $\iCat^{\Sp}$ denotes the $\infty$-category of $\infty$-categories enriched in the monoidal $\infty$-category $\Sp$ of spectra.  Since $\Sp$ carries the additional structure of a symmetric monoidal $\infty$-category, the $\infty$-category $\iCat^{\Sp}$ is also symmetric monoidal.  We can iterate this procedure further.  Let
\[
\Cat^{\Sp}_{(\infty, n)} = \iCat^{\iCat^{\phantom{X} \dotsm\,  \iCat^{\Sp}}}
\]
be the $\infty$-category of $\infty$-categories enriched in $\Cat^{\Sp}_{(\infty, n - 1)}$, where $\Cat^{\Sp}_{(\infty, 1)} = \iCat^{\Sp}$ \cite[5.7.13]{GH}.  We think of objects of $\Cat^{\Sp}_{(\infty, n)}$ as $(\infty, n)$-categories where the collection of $n$-morphisms between a pair of $(n - 1)$-morphisms forms a spectrum.  

\medskip

There is not currently a well-developed theory of weighted limits and weighted colimits in enriched $\infty$-categories, but we proceed as if there were, using the corresponding notions from ordinary enriched category theory as a guide for our intuition.  We write $\St^{(n)}$ for the full subcategory of $\Cat^{\Sp}_{(\infty, n)}$ consisting of those objects admitting all $\Cat^{\Sp}_{(\infty, n - 1)}$-weighted colimits.  Assuming that this is a sensible notion, there is a localization adjunction
\begin{equation}\label{eq:L^n}
\xymatrix{
\Cat^{\Sp}_{(\infty, n)} \ar[r]<.5ex>^-{L^{(n)}} & \St^{(n)} \ar[l]<.5ex>
}
\end{equation}
where the left adjoint freely adds all such colimits and the right adjoint is the inclusion.  The case $n = 1$ was studied by Blumberg-Gepner-Tabuada \cite{BGT} and takes the form
\[
\xymatrix{
\iCat^{\Sp} \ar[r]<.5ex>^-{L^{(1)}} & \St \ar[l]<.5ex>,
}
\]
where $\St = \St^{(1)}$ is the full subcategory of $\iCat^{\Sp}$ consisting of those spectrally enriched $\infty$-categories admitting small $\Sp$-weighted colimits, i.e. the cocomplete stable $\infty$-categories.  This explains the notation $\St^{(n)}$.  

\medskip
Suppose that $R$ is an $E_{\infty}$ ring spectrum.  There is an object $\mathfrak{b} R$ of $\iCat^{\Sp}$ with a single object $\bullet_{0}$ whose endomorphism spectrum is $\Hom(\bullet_{0}, \bullet_{0}) = R$.  
The image of the object $\mathfrak{b} R$ under the localization functor $L^{(1)}$ is the presentable stable $\infty$-category
\[
L^{(1)} \mathfrak{b} R = \Mod_{R}
\]
of $R$-modules.  The functor $L^{(1)}$ admits a symmetric monoidal structure and we write $(\Mod_{R}, \otimes)$ for the induced symmetric monoidal structure on the image of $\mathfrak{b} R$.  For the purposes of defining $K$-theory, we restrict to the full subcategory $(\Mod_{R}, \otimes)^{\circ}$ of dualizable $R$-modules under the symmetric monoidal structure.  It is here that we use the commutativity assumption.  We could equivalently restrict to the subcategory of perfect $R$-modules, meaning the objects which span the thick subcategory generated by $R$ after passage to the homotopy category.  It is the dualizability notion that we will generalize below, so we concentrate on that here.  

The $(\infty, 1)$-category $(\Mod_{R}, \otimes)^{\circ}$ is pointed and admits finite colimits, so we may use the $\infty$-categorical version of Waldhausen's $K$-theory construction \cite{BGT, Bar} to define its algebraic $K$-theory:
\[
K(R) = K(\Mod_{R}) := K((L^{(1)} \mathfrak{b} R, \otimes)^{\circ}).
\]
Notice that ``$K(R)$'' and ``$K(\Mod_{R})$'' are synonyms.  Soon we will consider the $K$-theory of other module categories, but in that case we will never write the $K$-theory in terms of the underlying ``ring'', only in terms of the category used to construct the $K$-theory.

\medskip
When $n = 2$, adjunction \eqref{eq:L^n} takes the form 
\[
\xymatrix{
\Cat^{\Sp}_{(\infty, 2)} = \iCat^{\iCat^{\Sp}} \ar[r]<.5ex>^-{L^{(2)}} & \St^{(2)} \ar[l]<.5ex>
}
\]
where the left adjoint $L^{(2)}$ freely adjoins all $\iCat^{\Sp}$-weighted colimits.  
There is an object $\mathfrak{b}^2 R$ of $\iCat^{\iCat^{\Sp}}$ which has a single object $\bullet_0$ and a single $1$-morphism $\bullet_1$ whose endomorphism spectrum is $\Hom(\bullet_{1}, \bullet_{1}) = R$.  The image of $\mathfrak{b}^2 R$ under the localization functor $L^{(2)}$ is the $(\infty, 2)$-category
\[
L^{(2)} \mathfrak{b}^2 R = \Mod_{\Mod_{R}}
\]
of $\Mod_{R}$-module categories.  An object of $\Mod_{\Mod_{R}}$ may be described as a presentable stable $\infty$-category $\cC$ equipped with a suitable action $\Mod_{R} \otimes \cC \arr \cC$ of $\Mod_{R}$, defined with respect to the symmetric monoidal structure $\otimes$ on stable presentable $\infty$-categories.  When considering dualizable $\Mod_{R}$-modules, as we will do below, it is natural to restrict to the action of the subcategory $\Mod_{R}^{\circ}$ of perfect $R$-modules.  If $\cA$ and $\cB$ are $\Mod_{R}$-module categories, the $(\infty, 1)$-category of morphisms from $\cA$ to $\cB$ is the category of exact $R$-linear functors 
\[
\Mod_{\Mod_{R}}(\cA, \cB) \simeq \Fun_{R}^{\ex}(\cA, \cB).
\]

So far, the constructions we have made would still make sense if $R$ were an $E_{2}$ ring spectrum, but as in the $n = 1$ case, we now restrict to the full subcategory spanned by the dualizable objects.  The spectral $(\infty, 2)$-category $\Mod_{\Mod_{R}}$ inherits a symmetric monoidal structure by virtue of the symmetric monoidal structure on the functor $L^{(2)}$.  We write $(L^{(2)}\mathfrak{b}^2 R, \otimes)^{\circ} \to L^{(2)}\mathfrak{b}^2 R$ for the inclusion of the subcategory spanned by the fully dualizable objects, as defined by Lurie \cite[2.3.19]{lurie_cobordism}.  This has the effect of discarding 1-morphisms that do not admit adjoints in the underlying homotopy bicategory, then discarding objects that are not dualizable in the underlying symmetric monoidal homotopy category.  We define the $K$-theory of $\Mod_{R}$-module categories to be the Waldhausen $K$-theory of the $(\infty, 1)$-category truncation of the $(\infty, 2)$-category of the fully dualizable objects of $\Mod_{\Mod_{R}}$
\[
K(\Mod_{\Mod_{R}}) := K(\iota_1(L^{(2)} \mathfrak{b}^2 R, \otimes)^{\circ}).
\]
In the case of symmetric monoidal bicategories with no higher morphisms, the analogous $K$-theory functor was constructed explicitly by Osorno \cite{Osorno}.  

When $R = Hk$ is the Eilenberg-MacLane spectrum associated to a field $k$, the $(\infty, 2)$-category $\Mod_{\Mod_{H k}}$ is equivalent to the $(\infty, 2)$-category of $k$-linear dg-categories.  The fully dualizable objects of $\Mod_{\Mod_{H k}}$ are precisely the smooth and proper dg-categories \cite{TVa, AG14}, and $\pi_0 K(\Mod_{\Mod_{Hk}})$ is isomorphic to the secondary $K$-theory $K_0^{\{2\}}(k)$ of the field $k$ defined by To\"en \cite{toen, tabuada}.  We also expect that $K(\Mod_{\Mod_{H k}})$ is related to the $K$-theory of varieties studied by Campbell-Wolfson-Zakharevich \cite{JAC, CWZ}.

\begin{rem}\label{2vect_remark}
A slight variation allows us to consider 2-vector spaces instead of dg-categories.  We replace $\Mod_{R}$ with the category $\Vect_{\C}$ of finite dimensional complex vector spaces, considered as an $(\infty, 1)$-category enriched in spectra by delooping the usual enrichment in topological abelian groups.  In this case, $\Mod_{\Vect_{\C}}$ is the $(\infty, 2)$-category of $\Vect_{\C}$-module categories.  In particular, the 2-category $2\Vect_{\C}$ of 2-vector spaces, as defined by Kapranov-Voevodsky \cite{KV}, embeds in the subcategory $(\Mod_{\Vect_{\C}}, \otimes)^{\circ}$ of dualizable objects, and $K(\Mod_{\Vect_{\C}})$ is equivalent to the bicategorical $K$-theory of 2-vector spaces $K(2\Vect_{\C})$ studied by Baas-Dundas-Rognes \cite{baas-dundas-rognes, Osorno}.  
\end{rem}

Instead of applying $L^{(2)}$ to $\mathfrak{b}^2 R$ , we could add colimits in a two step process.  Write $\iCat^{\St}$ for the $\infty$-category of $\infty$-categories enriched in stable $\infty$-categories.  There is a localization adjunction
\[
\xymatrix{
\iCat^{\iCat^{\Sp}} \ar[r]<.5ex>^-{\mathfrak{b} L^{(1)}} & \iCat^{\St} \ar[l]<.5ex>
}
\]
where the left adjoint applies $L^{(1)}$ to the hom objects.  In other words, we've only added colimits at the top level.  In the case of $\mathfrak{b}^2 R$, we get the category
\[
(\mathfrak{b} L^{(1)}) (\mathfrak{b}^2 R) = \mathfrak{b} \Mod_{R}
\]
with a single $0$-cell whose endomorphism object is the stable $\infty$-category of $R$-modules.  Notice that composition in this category uses the symmetric monoidal structure $\otimes$ on $\Mod_{R}$.  Next, we apply $K(-)$, taking care to first restrict to the dualizable objects.  Formally, we are using the fact that $\mathfrak{b} R$ is a commutative monoid in $\iCat^{\Sp}$ and that the functors $L^{(1)}$ and $K$ induce functors on categories enriched in commutative monoid objects.  We have now associated to $\mathfrak{b}^2 R$ the spectral category 
\[
\mathfrak{b}K(R) = (\mathfrak{b}K)(\iota_{1} ((\mathfrak{b}L^{(1)}) \mathfrak{b}^2 R, \otimes)^{\circ})
\]
with one object and $K(R)$ as its endomorphism spectrum.  Here $\iota_1$ denotes the truncation functor $\mathfrak{b}\iota_0$ taking an $(\infty, 2)$-category to the underlying $(\infty, 1)$-category that removes the non-invertible 2-morphisms.
After applying the same procedure one categorical level down by adding all spectral colimits, restricting to dualizable objects and taking $K$-theory,  we get the twice-iterated algebraic $K$-theory spectrum
\[
K(K(R)) = K((L^{(1)} \mathfrak{b} K(R), \otimes)^{\circ}).
\]

\medskip

The $K$-theory of $\Mod_{R}$-modules $K(\Mod_{\Mod_{R}})$ and the iterated $K$-theory $K(K(R))$ both map to the spectral enhancement $K^{\{2\}}(R)$ of To\"en's secondary $K$-theory defined by Hoyois-Scherotzke-Sibilla \cite{HSS}, as we now explain.  They consider a variant of the Blumberg-Gepner-Tabuada category of noncommutative motives \cite{BGT} that is the recipient of the universal additive invariant of stable $\Mod_{R}^{\circ}$-module categories.  More precisely, there is a functor
\[
\cU \colon \langle\text{$\Mod_{R}^{\circ}$-modules in $\Cat_{\infty}^{\mathrm{perf}}$}\rangle \arr \Mot_{R} 
\]
from the $\infty$-category of small, stable, idempotent complete $\Mod_{R}^{\circ}$-module categories to the $\infty$-category of $R$-linear motives that preserves filtered colimits, preserves zero objects, sends split exact sequences to cofiber sequences, and is universal among functors with these three properties.  The $(\infty, 1)$-category $\Mot_{R}$ and the functor $\cU$ admit compatible symmetric monoidal structures, and we define the $K$-theory of $\Mot_{R}$ to be the Waldhausen $K$-theory of the full subcategory of $\Mot_{R}$ generated under finite colimits and retracts by the images of dualizable objects under $\cU$:
\[
K(\Mot_{R}) := K((\Mot_{R}, \otimes)^{\circ}).
\]
$K$-theory becomes corepresentable after passing to the $\infty$-category $\Mot_{R}$, meaning that if $\cA$ and $\cB$ are $\Mod_{R}^{\circ}$-module categories, and $\cA$ is compact, then there is a natural equivalence of spectra
\[
\Mot_{R}(\cU(\cA), \cU(\cB)) \simeq K(\Fun^{\ex}_{R}(\cA, \cB)).
\]
When $\cA$ and $\cB$ are both fully dualizable as $\Mod_{R}$-module categories, which in particular implies that they are compact \cite[4.19]{HSS}, then their associated motives are dualizable in $\Mot_{R}$.  Applying the natural transformation $\iota_0 \arr K$ from the groupoid core functor to Waldhausen $K$-theory to the spectrally enriched $(\infty, 1)$-categories $\Fun^{\ex}_{R}(\cA, \cB)$ gives a map
\[
\iota_1\Mod_{\Mod_{R}}(\cA, \cB) \simeq \iota_{0}\Fun^{\ex}_{R}(\cA, \cB) \arr K(\Fun^{\ex}_{R}(\cA, \cB)) \simeq \Mot_{R}(\cU(\cA), \cU(\cB))
\]
which assembles into a symmetric monoidal functor of $(\infty, 1)$-categories
\[
\iota_{1} (\Mod_{\Mod_{R}})^{\circ} \arr \Mot_{R}^{\circ}.
\]
Applying Waldhausen $K$-theory gives a map of $E_{\infty}$ ring spectra
\begin{equation}\label{map1}
K(\Mod_{\Mod_{R}}) \arr K(\Mot_{R}).
\end{equation}

On the other hand, the motive associated to the category $\Mod_{R}^{\circ}$ of perfect $R$-modules is the unit of the symmetric monoidal category $\Mot_{R}$ of $R$-linear motives, and so the thick subcategory that it generates is equivalent to the category of modules over the endomorphism ring 
\[
\Mot_{R}(\cU(\Mod_{R}^{\circ}), \cU(\Mod_{R}^{\circ})) \simeq K(\Mod_{R}^{\circ}) = K(R).
\]
The inclusion of this thick subcategory induces a map of $E_{\infty}$ ring spectra
\begin{equation}\label{map2}
K(K(R)) \arr K(\Mot_{R}),
\end{equation}
as observed by Hoyois-Scherotzke-Sibilla \cite[6.23]{HSS}.

\begin{question}\label{image_compare_question}
To what extent do the images of the maps \eqref{map1} and \eqref{map2} coincide in $K(\Mot_{R})$?
\end{question}

When $R = Hk$ is the Eilenberg-MacLane spectrum associated to a field $k$, the work of Baas-Dundas-Richter-Rognes \cite{baas-dundas-richter-rognes} implies that $K(K(k))$ is equivalent to the $K$-theory of the full subcategory of $\Mod_{\Mod_{Hk}}$ spanned by the free 2-vector spaces of finite rank.  Equivalently, these are the dg-categories with a single object and the free algebras $k^n$ as the ring of morphisms.  Similarly, if we replace $\Mod_{R}$ with $\Vect_{\C}$, as in Remark \ref{2vect_remark}, then $K(ku)$ is equivalent to the $K$-theory $K(2\Vect_{\C}) \simeq K(\Mod_{\Vect_{\C}})$ of 2-vector spaces.

\medskip

We would like to ask a similar question about the iterated algebraic $K$-theory spectrum $K^{(n)}(R)$ for higher values of $n$.  Let $\mathfrak{b}^n R \in \Cat^{\Sp}_{(\infty, n)}$ be the $(\infty, n)$-category enriched in spectra with a single $k$-morphism for $0 \leq k \leq n - 1$ and for which the endomorphism spectrum of the $(n - 1)$-morphism is the ring $R$.  Let
\[
\Mod^{(n)}_{R} = L^{(n)} \mathfrak{b}^n R = \Mod_{\Mod_{ \dotsm \Mod_{R} } }
\]
denote the image of $\mathfrak{b}^n R$ under the localization functor $L^{(n)}$.  The fully dualizable part
\[
\Mod^{\circ}_{\Mod_{ \dotsm \Mod_{R}}} = (\Mod^{(n)}_{R}, \otimes)^{\circ}
\]
of the symmetric monoidal $(\infty, n)$-category $\Mod^{(n)}_{R}$ is our higher categorical analog of the category of perfect $R$-modules, and in the case of $R = Hk$ we consider this to be a reasonable definition of the $(\infty, n)$-category of $n$-chain complexes over $k$.  Similarly, we define the $(\infty, n)$-category of complex $n$-vector spaces to be the symmetric monoidal $(\infty, n)$-category
\[
n\Vect_{\C} = (L^{(n)} \mathfrak{b}^{n - 1} \Vect_{\C}, \otimes)^{\circ} = \Mod^{\circ}_{\Mod_{ \dotsm \Vect_{\C} } }.
\]

\begin{rem}
We do not know how to construct the appropriate analog of the category of non-commutative $R$-local motives 
$\Mot_{R}$ that would accept maps from the spectra $K^{(n)}(R)$ and
 $K((\Mod^{(n)}_{R}, \otimes)^{\circ})$, and we do not know how to ask the analog of Question \ref{image_compare_question} when $n > 2$ (compare with \cite[6.22]{HSS}).  Our results on T-duality for $\KA_n$-algebras suggest that a good understanding of the relationship between $K^{(n)}(R)$ and $K((\Mod^{(n)}_{R}, \otimes)^{\circ})$ would connect $T$-duality for $S^{q}$-bundles and the theory of $n$-vector spaces for $n = 2q + 1$.
\end{rem}

\bibliography{biblio}

\end{document}